\documentclass[11pt]{amsart}
\usepackage{xcolor}

\usepackage[
	pagebackref,  
	breaklinks=true,
	bookmarks=false,
	colorlinks,
	linkcolor=blue,
	citecolor=red,
	urlcolor=red,
]{hyperref}

\newtheorem{theorem}{Theorem}[section]
\newtheorem{lemma}{Lemma}[section]
\newtheorem{proposition}{Proposition}[section]

\theoremstyle{definition}
\newtheorem{definition}{Definition}[section]
\theoremstyle{remark}
\newtheorem{remark}{Remark}[section]
\numberwithin{equation}{section}
\newtheorem{corollary}{Corollary}[section]

\newtheorem{problem}{Problem}[section]

\numberwithin{equation}{section}

\newcommand{\C}{\mathbb C}
\newcommand{\Z}{\mathbb Z}
\newcommand{\R}{\mathbb R}
\newcommand{\CP}{\mathbb C\mathbb P}

\newenvironment{dedication}
  {\thispagestyle{empty}
   \itshape             
   \raggedleft          
  }
  {\par 
  }

\begin{document}

 \title[Singular fibrations]{Singular fibrations over surfaces}
\author[L.Funar]{Louis Funar}
\address{Institut Fourier BP 74, UMR 5582, Laboratoire de Math\'ematiques, 
Universit\'e Grenoble Alpes, CS 40700, 38058 Grenoble cedex 9, France}
\email{louis.funar@univ-grenoble-alpes.fr}

\maketitle
\begin{dedication}
Dedicated to Norbert A'Campo on the occasion of his 80th birthday
\end{dedication}

\begin{abstract}
Singular fibrations generalize achiral Lefschetz fibrations of 4-manifolds over surfaces 
while sharing some of their properties. For instance, relatively minimal singular fibrations are determined by their monodromy. We  explain how to construct 
examples of singular fibrations with a single singularity and Matsumoto's construction of singular fibrations of the sphere $S^4$. Previous results of Hirzebruch and Hopf on 2-plane fields with finitely many singularities are outlined in connection with the work of Neumann and Rudolph on the Hopf invariant. Eventually, we prove that closed orientable 4-manifolds with large first Betti number and vanishing second Betti number  do not admit singular fibrations.  

\vspace{0.1cm}
\noindent MSC Class: 57 R 45, 57 R 70. 


\end{abstract}


\section{Introduction}

The aim of this chapter is to study smooth maps 
$f:M\to N$ between a compact connected oriented 4-manifold $M^4$ and a compact oriented surface $N^2$ having only finitely many 
critical points. 
We start with: 

\begin{definition}
A  singularity $p$ of the smooth map $f:M\to N$ is {\em cone-like} 
if $p$ admits a cone neighborhood in the corresponding fiber $V=f^{-1}(f(p))$, i.e. there exists some closed manifold 
$L\subset V-\{p\}$ and a neighborhood $N$ of $p$ in $V$ which is homeomorphic to 
the cone $C(L)$ over $L$. Recall that $C(L)=L\times (0,1]/L\times \{1\}$. 
\end{definition}

Church and Timourian proved  (see \cite{CT2} and another proof in \cite{F}) 
that all isolated singularities in codimension at most $2$ are cone-like. 
In particular, isolated critical points of smooth maps $f:M^4\to N^2$ are cone-like. 
As Takens pointed out in \cite{Takens}, this is not anymore true in codimension 3 or higher.

\begin{definition}
A  cone-like singularity of the smooth map $f:M\to N$ is {\em regular} 
if it admits an adapted neighborhood (see section \ref{fiberedlinks}) which is diffeomorphic to a disk. 
\end{definition}

By a result of King (see \cite{King,King2}) unless $\dim M\in\{4,5\}$ 
all cone-like singularities are regular. When $\dim M=4$ the general  
theory developed by King only shows that we can find 
contractible adapted neighborhoods.

\begin{definition}
A smooth map $f:M\to N$ is {\em a singular fibration} 
if it has only finitely many critical points, all of them being regular and cone-like.  
\end{definition}

We will focus in this chapter on singular fibrations $M^4\to N^2$. 
In this case, around each critical point the map is 
described by (a generalization of) the local Milnor fibration of a 
germ with isolated singularity $(\R^4,0)\to (\R^2,0)$. 
Such a singularity is topologically locally determined by a fibered link $K\subset S^3$, called the 
{\em local link} of the singularity, see 
Section 2.1 for details. 

The simplest nontrivial singularity is given by a complex Morse map $f:(\C^2,0)\to (\C,0)$  
\[ f(z_1,z_2)=z_1^2+z_2^2.\]
The corresponding  local link is the positive Hopf link, having linking number $+1$.  
Maps $M^4\to N^2$ admitting a local system of real coordinates  which is globally 
coherently oriented such that the restriction to every chart is topologically equivalent to a complex Morse map are called {\em Lefschetz fibrations}. Note that we do not require the atlas to be a global complex coordinate system and in particular Lefschetz fibrations  need not  be holomorphic.    

Taking the complex conjugate of a coordinate we obtain the map $f'(z_1,z_2)=\overline{z}_1^2+z_2^2$ which is equivalent to a complex Morse map by an {\em orientation reversing} 
change of coordinates. Its local link at the origin is the mirror image of the positive Hopf link, and hence the negative Hopf link with linking number $-1$. We call {\em quadratic} singularities 
those critical points which are locally equivalent to either $f$ or $f'$, by orientation 
preserving changes of coordinates.

Recall that a smooth map $f:M^4\to N^2$  whose singularities are quadratic is called an {\em achiral Lefschetz fibration}. Thus achiral Lefschetz fibrations correspond to maps whose singularities have local links equivalent to Hopf links.

Fibered links arising as local links of  complex polynomials or holomorphic functions with isolated singularities were intensively studied. They are well understood, as all of them are cabled torus links. 

Whether all fibered links can arise as local links of real polynomials or real analytic functions 
with isolated singularities seems presently unknown.  
According to Looijenga (\cite{Loo}), if $K\subset S^3$ is an arbitrary link, then the connected sum $K\# K$ is the link of a real polynomial isolated singularity $(\R^4,0)\to (\R^2,0)$.  Moreover, any fibered link $K$ is the link of a semi-algebraic local isolated singularity, whose components are of the form $P(x)+|x| Q(x)$, where $P, Q$ are real polynomials and $| \; |$ denotes the norm.

The present chapter aims at addressing some of the following questions below. 

\begin{problem}
Can we classify singular fibrations? 
\end{problem}

If we fix the topological type of the generic fiber and the number of branch points (i.e. singular values), 
then a first step towards an answer, in purely algebraic terms, is provided by Proposition \ref{classification}. This result is a 
high-dimensional analog of the Hurwitz classification of 
ramified coverings between surfaces.  Nevertheless, it is rather ineffective as a 
description of diagonal conjugacy classes of tuples of elements 
in the mapping class group is  presently lacking.  In particular, it does not lead to insights on:

\begin{problem}\label{existence}
Which closed  oriented 4-manifolds admit singular fibrations over surfaces? 
\end{problem}

Harer proved in \cite{Harer} that a 4-dimensional manifold with boundary has a handlebody decomposition with handles of index $\leq 2$ if and only if it admits an achiral Lefschetz fibration over the 2-disk. An analog of Harer's result was proved recently for the nonorientable case by Miller and Ozbagci (\cite{MO}). 
In the same vein, Matsumoto proved in \cite{Mat2} that a closed smooth simply connected 4-manifold which admits a handlebody decomposition without 1-handles and with at most one 3-handle admits a 
continuous torus singular fibration over the 2-sphere which is smooth everywhere except at one point. 
Moreover, Etnyre and Fuller  (\cite{EF}) proved that for any smooth closed simply connected orientable 4-manifold $M^4$, the connected sum 
$M\# S^2\times S^2$ admits an achiral Lefschetz fibration over $S^2$. 

On the other hand, Gompf and Stipsicz proved that there exist 4-manifolds $M^4$ which do not admit achiral Lefschetz 
fibrations over $S^2$, 
for instance $\#_m S^1\times S^3$ when $m\geq 2$ (see \cite{GS}, Section 8.4).
We will show that these manifolds do not admit singular fibrations either, see Proposition \ref{notfiber}. The proof uses in an essential way the analysis made by 
Hirzebruch and Hopf in \cite{HH} of the oriented 2-plane tangent fields with finitely 
many singularities on 4-manifolds. 

\begin{problem}\label{charconj}
Suppose that there exists a singular fibration $M^4\to N^2$ between closed orientable manifolds.  
Does $M^4$ admit an achiral Lefschetz fibration over $N^2$ or another surface?
\end{problem}

We say that a singularity is {\em $\pm$-holomorphic} if it is locally topologically equivalent to a 
holomorphic germ with an isolated singularity. We emphasize that the local homeomorphism 
need not preserve the orientation. 
It is well-known that the problem above has an affirmative answer when the singularities are $\pm$-holomorphic, see Proposition \ref{holomorphic}. 

Assume that the isotopy type of the generic fiber is fixed. Then we have to look 
over the set of singular fibrations $X^4\to D^2$ having the same (vertical) boundary fibrations over $S^1$. Such singular fibrations are called {\em shell} equivalent.

\begin{problem}\label{unfold}
Describe singular fibrations up to shell equivalence. 
\end{problem}

Some invariants of shell equivalence classes are derived from the Hopf invariants, see 
Corollaries \ref{shellindex1} and \ref{shellindex2}.  
The shell equivalence of singular fibrations $D^4\to D^2$, was 
considered in \cite{NR}, where it was called 
{\em unfolding} for the respective collections of local links. 
The work of Neumann and Rudolph from \cite{NR,NR2,NR3}
provides fibered links which do not unfold to Hopf 
links, see also \cite{Inaba}.

Let $\varphi(M,N)$ denote the minimal number of critical points
of a smooth map between the manifolds $M$ and $N$. Although there is a large supply of results concerning the minimal number of singular fibers and hence of  singular points for Lefschetz fibrations, see \cite{KO,OO}, there is little known about $\varphi(M^4,S^2)$. 
Note that the $\varphi(M^4, S^2)$ can be strictly smaller than the 
number of critical points in an achiral Lefschetz fibration. 
Indeed $\varphi(S^4,S^2)=1$, by the work \cite{Mat,Mat2} of Matsumoto, while  
the  (minimal) number of critical points in an achiral Lefschetz 
fibration of $S^4$ over $S^2$ is $2$. 
We set $\varphi_{s}(M^4,N^2)$ for the minimal 
number of critical points of a singular fibration $M^4\to N^2$. 
Similar invariants counting cone-like critical points have been studied 
in higher dimensions in \cite{AFK,FP} and they are conjecturally the same as $\varphi$.   
We don't know under which conditions singularities can merge together to produce a higher singularity and a smaller number of critical points. In particular, the following 
might be relevant:

\begin{problem}\label{varphi}
Can we characterize 4-manifolds with $\varphi_{s}(M^4,N^2)=1$, for some $N^2$?  
\end{problem}

We will provide necessary conditions on the fundamental groups of manifolds with the property 
$\varphi_{s}(M^4,S^2)=1$ and derive that $\varphi_{s}$ can take values larger than 1. 
We further give a satisfactory answer to the problem above in the case of homology spheres, see Propositions \ref{onecritical} and \ref{hspheres}. 

Note that there is also a group-theoretical invariant derived out of $\varphi_{s}$. For any 
finitely presented group $G$ there exists at least one closed smooth 4-manifold with 
$\pi_1(M^4)$ isomorphic to $G$. The set of all possible values taken by $\varphi_{s}(M^4, S^2)$,  
in particular  the lowest value, is an invariant 
of $G$. However, it is difficult to compute. For instance it is unknown whether:

\begin{problem}
Is it true that any smooth closed simply connected oriented 4-manifold has a singular fibration over some surface?
\end{problem}

The starting point of this article is the work of A'Campo on real morsifications of isolated singularities of planar curves (see \cite{AC1}). 
We were strongly influenced by the papers of Neumann and Rudolph (\cite{NR,NR2,NR3,Rud,Rud2}) on the  Hopf invariant,  the 
Hirzebruch--Hopf results on 2-plane fields on 4-manifolds 
(\cite{HH}) and  Matsumoto's work on constructing singular fibrations on 4-manifolds (\cite{Mat0,Mat,Mat2}).

{\bf Acknowledgments.} We are indebted to Norbert A'Campo, 
Fran\c{c}ois Dahmani, Pierre Dehornoy,  Nicolas Dutertre, 
Christine Lescop, Burak Ozbagci, Andras Stipsicz and Mihai Tib\v{a}r for useful discussions and suggestions.

\section{Fibered links and the description of singular fibrations}

\subsection{Fibered links and local models for isolated singularities}\label{fiberedlinks}
Recall that the isotopy class of the oriented 
link $K\subset X^{3}$ with a trivial 
normal bundle  is called {\em generalized Neuwirth--Stallings fibered} (or $(X^{3}, K)$ 
is a  generalized Neuwirth--Stallings pair) 
if, for some trivialization $\theta: N(K)\to K\times D^2$ of the tubular 
neighborhood $N(K)$ of $K$ in $X^{3}$,  
the fiber bundle $\pi\circ \theta: N(K)-K \to S^{1}$ admits an extension 
to  a smooth fiber bundle 
$f_K:X^{3}- K\to S^{1}$. Here $\pi:K\times (D^2-\{0\})\to S^{1}$ is the 
composition of the radial 
projection $r: D^2-\{0\}\to S^{1}$ with the second factor projection. This is equivalent to the condition that the closure of every fiber is its compactification  by the binding link.
The data $\mathcal E=(X^{3}, K, f_K, \theta)$ is then called an {\em open book decomposition} with binding $K$ while $K$ is called a {\em fibered link}. We will frequently use the notation 
$f_K$ for the induced fibration $f_K:X^{3}- N(K)\to S^{1}$, whose fiber is now compact. 
When $X^{3}=S^{3}$ we have the classical notions of Neuwirth--Stallings fibrations and pairs. 

An {\em adapted neighborhood} around a cone-like isolated  critical point of a 
map $f:M^4\to N^2$ is a compact manifold neighborhood $Z^{4}\subset M^4$ containing it 
with the following properties:

\begin{enumerate}
\item $f$ induces a proper map $f:Z^4\to D^4$ onto a disk $D^2\subset N^2$; 
\item $f^{-1}(x)$ is transversal to $\partial Z^4$ for $x\in{\rm int}(D^2)$ and 
$E=f^{-1}(S^{1})\cap Z^4\subset \partial Z^4$;   
\item Let $K=\partial Z^4\cap f^{-1}(0)$ and $N(K)=\partial Z^4 \cap f^{-1}(D_0^2)$ 
a tubular neighborhood of $K$ within $\partial Z^4$, 
endowed with the trivialization $\theta$ induced by $f$, where $D_0^2\subset D^2$ is a small disk around $0$. Then the composition $f_K=r\circ f: \partial Z^4- f^{-1}(0)\to S^{1}$ 
of the radial projection $r$ with $f$ is a locally trivial fibre bundle, so that 
the data $(\partial Z^4, K, f|_{K}, \theta)$ is an open book decomposition.    
\end{enumerate}

King proved in \cite{King} that all cone-like isolated singularities admit 
contractible adapted neighborhoods. However, we don't know whether 
one can always choose adapted neighborhoods diffeomorphic to a disk.  
For this reason we restrict ourselves from now on to those cone-like singularities 
which are regular.

Recall now from \cite{Loo,KN,Rud}  that an open book decomposition $\mathcal E=(S^{3}, K, f_K, \theta)$ gives rise to a local  
isolated singularity $\psi_{\mathcal E}:(D^4,0)\to (D^2,0)$  by means of the formula:
\[\psi_{\mathcal E}(x)=\left\{\begin{array}{cl}
\lambda(\|x\|)f_K\left(\frac{x}{||x||}\right), & {\rm if} \; \frac{x}{||x||}\not\in N(K); \\
\lambda\left( \|x\|\cdot \left\|\pi_2\left(\theta\left(\frac{x}{\|x\|}\right)\right)\right\|\right)f_K\left(\frac{x}{||x||}\right), & {\rm if} \; \frac{x}{||x||}\in N(K); \\
0, & {\rm if }\;  x=0,
\end{array}
\right.
\]
where $\pi_2:K\times D^2\to D^2$ is the projection on the second factor and $\lambda:[0,1]\to [0,1]$ is any smooth 
strictly increasing map sufficiently flat at $0$ and 1 such that $\lambda(0)=0$ and 
$\lambda(1)=1$. 

Although $\psi_{\mathcal E}$ is not uniquely defined by this formula, all 
germs obtained this way are topologically equivalent. This is a direct consequence of the characterization of cone-like isolated singularities due to King 
(see \cite{King}, Thm. 2 and \cite{King2}, Thm. 1).  Moreover, they are also smoothly equivalent 
by germs of diffeomorphisms of $D^m\setminus\{0\}$ (see \cite{KN}, Thm. 1.10, for $k=2$).  
We call such $\psi_{\mathcal E}$ {\em local models} of isolated singularities.

If  $K$ is in generic position, namely the space generated by vectors in $\R^4$ with endpoints in $K$ coincides with 
the whole space $\R^4$, then  $(d\psi_{\mathcal E})_{0}=0$, i.e. $\psi_{\mathcal E}$ has rank $0$ at the origin.

By language abuse we will speak of fibered links $K$ as being links which admit an open book structure. We emphasize that, in general, $\psi_{\mathcal E}$ depends on the choice of the fibration $f_K$ and the trivialization $\theta$, not of the embedding of $K$ alone. However, it is specific to dimension 3 that the fiber $F\subset M^3\setminus K$ determines not just the binding $K=\partial F$ but also the fibration $f_K$ up to isotopy relative to $F$. This follows from a classical result of Waldhausen stating that homotopic fibrations are actually isotopic. It is actually known that the {\em oriented} fibered link $K\subset M^3$ in a homology sphere $M^3$ determines uniquely its open book fibration, up to isotopy (see \cite{GG}, p.99).

Looijenga in \cite{Loo} proved that a Neuwirth--Stallings pair $(S^{3},L, f_L, \theta)$  can be realized by a {\em real polynomial} map  
if $L$ is invariant and the open book fibration $f_L$ is equivariant with respect to the antipodal map.  
In particular, for any fibered link $K$ the connected sum $K\# K$ 
is a Neuwirth-Stallings pair isomorphic to the 
link of a real polynomial isolated singularity $\psi_K:(\R^4,0)\to (\R^2,0)$.

\subsection{Fiber connectedness} It was noted by Gompf and Stipsicz in \cite{GS} that singular fibrations behave in many respects like achiral Lefschetz fibrations, in particular we have the following  analog of the long exact sequence in homotopy: 
\begin{lemma}\label{lesh}
For any singular fibration $f:M^4\to N^2$ with generic fiber $F$ 
there is an exact sequence: 
\[ \pi_1(F)\to \pi_1(M^4)\to \pi_1(N^2)\to \pi_0(F)\to 0.\]
\end{lemma}
\begin{proof}
See \cite[Prop. 8.1.9 and Ex.8.1.10.(b)]{GS}. 
\end{proof}
In particular the image of $\pi_1(M^4)$ in $\pi_1(N^2)$ has always finite index.  
When the fiber $F$ is not connected, $f$ lifts to a singular fibration 
of $M^4$ to a finite cover of $N^2$ such that the induced map between fundamental 
groups is surjective and hence the generic fiber is now connected. 

Henceforth we only consider singular fibrations with 
{\em connected} generic fibers, throughout this chapter.

\subsection{Singular fibrations}\label{cutpaste}
The method used in \cite{FPZ,F,FP} to globalize a local picture was to glue together a patchwork of such local models 
to obtain maps $M^{4}\to N^{2}$ with finitely many 
critical points.  We obtain this way a complete description of singular fibrations in terms of algebraic data.  For the sake of simplicity consider below the case when $N^2=S^2$. 

Let  $K_j\subset S^3$ be  fibered knots whose associated open books have the fibers $F_j^{2}$
and the associated smooth maps $\psi_{K_j}:(D_j^4,0)\to (D^2,0)$. 
Consider some abstract surface $Z^{2}$ with $\partial Z^{2}=\sqcup_j K_j$. 
We now define the {\em expansion} of the fiber by means of $Z$ as follows.  
We glue together the copies $D_j^4$  of the $4$-disk with the product  
$Z\times D^2$ along part of their boundaries 
by identifying  $\sqcup_j N(K_j)$ with $\partial Z\times D^2$. 
The identification has to respect the trivializations  
$\sqcup_j N(K_j)\to D^2$  and hence one can take them to be 
the same as in the double construction. Note that $N(K_j)=K_j\times D^{2}$ and thus identifications respecting the trivialization correspond to homotopy classes $[\sqcup_j K_j, {\rm Diff}^+(D^2,\partial)]=1$.

We then obtain  a manifold with boundary $X=(\cup_j D_j^4)\cup (Z\times D^2)$
endowed with a smooth map $f_{(K_j)}: X\to D^2$ which restricts to 
$\psi_{K_j}$ on each disk $D_j^4$  and with the second projection on $Z\times D^2$. 
Then $f_{(K_j)}$ has finitely many critical points, all of them lying in 
the same  singular fiber over the origin $0\in D^2$.   
Its generic fiber is the surface $F$ obtained by 
gluing together  $(\cup_jF_j)\cup Z$. We call then 
the pair $(X, f_{(K_j)})$ a {\em disk block}. Its boundary $\partial X$ inherits 
a fibration $f_{(K_j)}|_{\partial X}: \partial X\to S^1$.

We say that a collection of fibrations over the circle  with the same fiber are {\em cobounding} in $N^2$ if they extend to some fibration over $N^2\setminus \sqcup_{i=1}^{p}D_i^2$. 
Let now $f_i:X_i\to D_i^2, 1\leq i\leq p$ be a set of disk blocks with the same fiber $F$ whose    
boundary fibrations are {\em cobounding}.  
We can then glue together the disk blocks and a fibration $E$ over  $N^2\setminus \sqcup_{i=1}^{p}D_i^2$ extending the boundary fibrations of disk blocks. We want the gluing diffeomorphisms to respect the fibration structure on the boundaries. Therefore we obtain       
a closed manifold $M(f_1,f_2,\ldots,f_p,E)$ endowed with a map with finitely many critical points into $N^2$. We will drop $E$ from the notation when its choice is implicit.  In most cases below $E$ will be a trivial fiber bundle.  
Moreover, 
the resulting manifold depends on the {\em gluing parameters}: for each circle 
$\partial D_i^2$ we have a (fixed) gluing diffeomorphism 
of the fiber $F$ over a base point and the homotopy
class of a loop in the space of diffeomorphisms of $F$ isotopic to the identity. 
Recall, after Earle and Eells (see \cite{EE}), that the space of diffeomorphisms isotopic to the identity is homotopy equivalent to $SO(3)$, if $F=S^2$, to a torus, if $F$ is a torus and 
contractible, when $F$ has higher genus. 
Note that the extension fibration $E$ is not unique, in general. However, when 
$N^2=S^2$, if an extension $E$ exists, then it is unique.

Conversely, every singular fibration $M^4\to N^2$ is equivalent to some 
singular fibration obtained by the previous construction, of the form  $M(f_1,f_2,\ldots,f_p, E)\to N^2$. 

We can replace each local model $\psi_{K_j}$ by left composing it with 
some diffeomorphism in ${\rm Diff}^+(D^2,\partial)$ sending 
the origin $0\in D^2$ into some point $p_j$, such that the points $p_j$ are distinct. 
Then the critical points of the new map $f_{(K_j)}:X\to D^2$ are now sitting in distinct 
fibers. 
If we are only interested in the diffeomorphism type of the 
4-manifold $M$ it  suffices therefore to consider the case when there is 
a single critical point in every singular fiber. In this situation each 
disk block map $f_i:X_i\to D_i^2$ 
is obtained from $\psi_{K_i}:D_i^4\to D_i^2$ by expanding the 
local Milnor fiber $F_i$ using boundary connected sums with some surface $Z_i$. 
In other terms the singular fibrations whose singular points have distinct images 
are completely determined by the fibration part $E$, the gluing parameters,  
the decompositions of the generic fiber $F$ as the boundary connected sums $F_i\cup Z_i$ and 
the monodromies of the open book decompositions associated to the knots $K_i$. 

Alternatively we can consider the fiber-sum of the disk blocks $f_i:X_i\to D_i^2$ 
and obtain a disk block $f:\#_i X_i\to \#_i D_i^2=D^2$, 
where $D^2\subset N^2$ is a disk containing all singular values. 
The manifold  $M(f_1,f_2,\ldots,f_p, E)$ is then obtained as the result of 
gluing together the disk block $X$ with the restriction of the fibration $E$ to 
$N^2\setminus D^2$.

We denote by $M(F)$ the {\em mapping class group} of the surface $F$. 
Choose now a base point $*$ on $S^2$ and a system of disjoint embedded arcs connecting it to 
base points on $\partial D_i^2$. This data permits to describe a fibration $E$ over 
$S^2\setminus  \sqcup_{i=1}^{p}D_i^2$ by means of the $p$-tuple of the monodromies $\phi_j$ 
along the loops issued from $*$ which follow the given arcs and encircle once 
clockwise the singular value in $D_j^2$. Note that $(\phi_j)\in M(F)^p$ is well-defined, up to 
diagonal conjugation by $M(F)$.

Now, a collection of fibrations over the circle is cobounding in $N^2$ if and only if 
the total monodromy $\phi_1\phi_2\cdots \phi_p$ is the product of $g$ commutators in $M(F)$, 
where $g$ is the genus of $N^2$. In particular, it is cobounding in $S^2$ if and only if 
the monodromy along a loop encircling all critical values is trivial:
\[ \phi_1\phi_2\cdots \phi_p=1\in M(F).\]

\subsection{Singular fibrations over $S^2$}\label{mono}
Throughout this section singular fibrations are required to have singular points 
with distinct images. 

\begin{definition}
Two singular fibrations $f,g:M^4\to S^2$, are {\em equivalent} if there exists some 
homeomorphisms $H:M^4\to M^4$ and $\varphi:S^2\to S^2$ such that $H\circ f=g\circ \varphi$. 
When $\varphi$ can be taken to be isotopic to the identity rel the branch locus, 
we say that the singular fibrations are {\em strongly equivalent}.
\end{definition}

Let $B\subset S^2$ denote the set of critical values of a singular fibration $f:M^4\to S^2$. 
Then the holonomy of the fiber bundle $f|_{M\setminus f^{-1}(B)}:M\setminus f^{-1}(B)\to S^2\setminus B$ provides us with a {\em monodromy homomorphism}: 
\[ \pi_1(S^2\setminus B)\to M(F),\]
which is well-defined up to conjugacy by $M(F)$.

Observe first that the monodromy 
of singular fibrations of oriented manifolds with generic fiber $S^2$ is trivial. 
In the study of Lefschetz fibrations one usually restricts to relatively minimal 
Lefschetz fibrations. In particular, no vanishing cycle bounds a disk in a generic fiber.  

By analogy, we say that a singular fibration whose generic fiber is a torus is {\em relatively minimal} if the complement 
$F\setminus {\rm int}(F_0)$ of the local Milnor fiber within the generic fiber $F$ 
does not have any disk component. The following extends the classification of 
genus one Lefschetz fibrations due to Moishezon (\cite{Moishe}):

\begin{proposition}\label{classificationg1}
Relatively minimal singular fibrations on $S^2$ with branch locus $B$, generic fiber 
a torus $T^2$ and surjective monodromy, up to strong equivalence are in one-to-one correspondence with the set of classes $M_{T^2}(S^2\setminus B)$ of surjective homomorphisms $\pi_1(S^2\setminus B)\to  SL(2,\Z)$ sending peripheral loops into  Dehn twists, modulo 
the conjugacy action by $SL(2,\Z)$. Furthermore, the classes of these 
relatively minimal singular fibrations up to equivalence are in one-to-one correspondence with the 
set  $\mathcal M_{T^2}(S^2\setminus B)$ of orbits of the mapping class group $M(S^2\setminus B)$ 
action on  $M_{T^2}(S^2\setminus B)$ by left composition.  
\end{proposition}
\begin{proof}
The proof is an adaptation of Moishezon's proof from \cite[Part II, section 2]{Moishe}. 
Since the fiber $F$ is a torus every local Milnor fiber $F_0$ is an annulus, 
since its  complement contains no disks. Therefore peripheral loops are sent into powers of Dehn twists. However, the monodromy of a fibered link whose fiber is an annulus 
is a (left or right) Dehn twist.

We have to prove next that a disk block $X\to D^2$ with torus fibers extends uniquely over 
the complementary disk of $D^2$ within $S^2$. Since the boundary fibration is trivial the gluing parameters correspond to homotopy classes of fiberwise diffeomorphisms of $S^1\times T^2$, namely to elements of the fundamental group $\pi_1({\rm Diff}_0(T^2))$ of the group of 
diffeomorphisms of $T^2$ which are isotopic to the identity. According to \cite{EE} the map which associates 
to $(m,n)\in\Z\oplus \Z$ the fiberwise diffeomorphism $L_{m,n}:S^1\times T^2\to S^1\times T^2$
\[ L_{m,n}(\theta,\alpha,\beta)=(\theta,\alpha+m\beta,\beta+n\theta), \; \theta\in S^1, (\alpha,\beta)\in S^1\times S^1=T^2,\]
induces an isomorphism between $\Z\oplus\Z$ and $\pi_1({\rm Diff}_0(T^2))$. 
If $f:M\to S^2$ is a singular fibration and $(m,n)\in\Z\oplus \Z$, we denote by
$f_{(m,n)}:M_{(m,n)}\to S^2$ the result of excising a regular neighborhood $D^2\times T^2$ of a 
regular fiber and gluing it back by twisting the boundary identification by $L_{m,n}$ and extending the restriction of the singular fibration $f$ by the second factor projection.
This is the so-called multiplicity one logarithmic transform in the case of Lefschetz fibrations. 

Let $\gamma$ be a based loop in $S^2\setminus B$ which encircles once the singular value $q$ 
and $D^2$ a disk centered at the base point. Then there exists an isotopy $\phi_t$ of $S^2$ which is the identity on a neighborhood of $q$ such that 
$\phi_1$ sends $D^2$ in itself, which drags $D^2$ along the curve $\gamma$. 
Let $\rho:\pi_1(S^2\setminus B)\to  SL(2,\Z)$ be the monodromy of the singular fibration $f$.  
The isotopy $\phi_t$ lifts to $M$ and provides a strong equivalence 
between the singular fibrations $f_{\mathbf m}$ and $f_{\rho(\gamma)\mathbf m}$, for any $\mathbf m\in \Z\oplus\Z$. Moreover, $f_{\mathbf m+\mathbf m'}$ is equivalent to 
 $\left(f_{\mathbf m}\right)_{\mathbf m'}$.  
Since $\rho$ is surjective, we obtain successive  equivalences between 
$f_{(m,n)}$, $f_{(m+n,n)}$, $\left(f_{(m+n,0)}\right)_{(0,n)}$, 
$\left(f_{(m+n,0)}\right)_{(-n,0)}$, $f_{(m,0)}$, $f_{(m,m)}$, 
$\left(f_{(m,0)}\right)_{(0,m)}$, $\left(f_{(m,0)}\right)_{(-m,0)}$  and $f_{(0,0)}$.
This proves the claim.       
\end{proof}

Before we proceed with the higher genus singular fibrations, we have to introduce more notation. 
Let $A(F)\subset M(F)$ denote the set of mapping classes arising as monodromies 
of fibered knots in $S^3$ with fibers which are subsurfaces contained in $F$. 

Let now $\phi\in A(F)$. The question to be addressed is whether $\phi$ determines 
uniquely the isotopy class of the subsurface $F_0\subset F$, which is 
the local Milnor fiber of the fibered link determined by $\phi$. 

To this purpose we define the {\em support} $sup(\phi)$ of the mapping class $\phi$ as follows. 
For any subsurface $F_0\subset F$, there is an induced homomorphism  
$\iota_{F_0}:M(F_0)\to M(F)$. The support $sup(\phi)$ is the smallest essential subsurface $F_0$, 
defined up to isotopy in $F$,  
with the property that $\phi\in \iota_{F_0}(M(F_0))$. 

Let $F_0$ and $F_1$ be two subsurfaces such that 
$\phi\in \iota_{F_i}(M(F_i))$, $i\in\{0,1\}$. We can use isotopies such that isotopic boundary components of $F_0$ and $F_1$  coincide and nonisotopic boundary components 
are transverse. Transverse boundary components are either disjoint 
or intersecting in the minimal possible number of points within their isotopy classes.  
If there are no intersections between $\partial F_0$ and $\partial F_1$, then 
$F_0\cap F_1$ is a subsurface strictly contained in both $F_i$ whose Euler characteristic 
must be strictly larger. Moreover, $\phi\in \iota_{F_0\cap F_1}(M(F_0\cap F_1))$. 
If two components $C_i\subseteq \partial F_i$ have nontrivial intersection 
number, then the arc $C_1\cap F_0$ should be essential in $F_0$. It follows that 
the surface $F_2$ obtained from $F_0$ by removing a regular neighborhood of 
the arc $C_1\cap F_0$ has strictly larger Euler characteristic. Moreover 
$\phi\in \iota_{F_2}(M(F_2))$. These arguments prove the existence of the support 
for any $\phi\in M(F)$. Note that $sup(\phi)$  is trivial if only if $\phi$ is the identity. 
  
We further define $M(F)^*\subset M(F)$ to be the set of those $\phi$ for which the support 
$sup(\phi)$ is connected. If $F_0$ is a connected subsurface and $\phi$ is the image 
by $\iota_{F_0}$ of a pseudo-Anosov mapping class, then $\phi\in M(F)^*$. On the other hand the 
product of  at least two Dehn twists along disjoint simple closed curves does not belong to 
$M(F)^*$.

The classification of singular fibrations $M^4\to S^2$ whose generic fiber $F$ has genus 
$g$  is then similar to that of Lefschetz fibrations, if we restrict the peripheral monodromies.  
Set 
$A(F)^*=M(F^*)\cap A(F)$. A singular fibration is {\em admissible} if 
its peripheral loop monodromies belong to $A(F)^*$.

Consider now a critical point of a singular fibration, whose local Milnor fiber is $F_0$ and the
generic fiber $F$ has genus $g\geq 2$. 
We say that the singular point is {\em  relatively minimal} if the complement 
$F\setminus {\rm int}(F_0)$ of the local Milnor fiber within the generic fiber $F$ 
is injective, namely it does not contain neither disk nor annular components. 
Moreover the singular fibration is  {\em  relatively minimal} if all its 
singular points are relatively minimal.

\begin{proposition}\label{classification}
Relatively minimal admissible singular fibrations on $S^2$ with branch locus $B$, 
generic fiber $F$ of genus $g\geq 2$, 
up to strong equivalence are in one-to-one correspondence with the set $M_{F}(S^2\setminus B)$ of homomorphisms $\pi_1(S^2\setminus B)\to  M(F)$ sending peripheral loops into $A(F)^*$ modulo 
the conjugacy action by $M(F)$. Furthermore, the classes of these 
relatively minimal singular fibrations up to equivalence are in one-to-one correspondence with the 
set  $\mathcal M_{F}(S^2\setminus B)$ of orbits of the mapping class group $M(S^2\setminus B)$ 
action on  $M_{F}(S^2\setminus B)$ by left composition.  
\end{proposition}
\begin{proof}
A singular fibration whose singular points have distinct images 
is obtained from the fiber sum of several disk blocks $f_i:X_i\to D_i^2$, by gluing the trivial 
bundle $F\times D^2$  in a fiberwise way. When $g\geq 2$, any diffeomorphism of $S^1\times F$ should send $S^1\times *$ to itself, up to isotopy and orientation reversing. 
Gluing $D^2\times F$ amounts therefore to  first adjoin a 2-handle along this curve with its natural framing and further adding 3-handles and 4-handles corresponding to the handle decomposition of $F$. According to \cite{LP} the second step can be done in a unique way. 
This implies that the 4-manifold  $M(f_1,f_2,\ldots,f_n)$ is uniquely determined, up to diffeomorphism, by the disk blocks involved.  
 
It remains to see that a disk block $f:X\to D^2$ with a single critical point is determined by 
its boundary fibration. This is generally not true, unless the singular point is 
relatively minimal, as we shall see later. 
The  monodromy of its boundary fibration over $S^1$ is an element 
$\phi\in A(F)^*$. We claim that $\phi$ already determines the local Milnor fiber $F_0$ and hence the fibered knot $K$. 
Let $F_0\subset F$ be the support of $\phi$, so that $\phi=\iota_{F_0}(\phi_0)$, for some 
$\phi_0\in M(F_0)$, to be called the restriction of $\phi$ to $F_0$. The restriction  
is uniquely determined because the map induced by the inclusion 
$M(F_0)\to M(F)$ is injective according to the main result of \cite{PR}.

Assume first that $\phi_0$ is the monodromy of a fibered link $K_0$, namely that 
the mapping torus $M_{\phi_0}$ is homeomorphic to $S^3-K_0$. Let $D^4$ be the 4-disk bounded by $S^3$.  We claim that by augmenting the support $F_0$ to a larger subsurface $F_0\subset F_1\subseteq F$ the corresponding expansion $\phi_1$ of $\phi_0$ by the identity on $F_1-F_0$ is not anymore a fibered link monodromy, 
unless $\Sigma=F_1\setminus {\rm int}(F_0)$ is a disjoint union of annuli, each annulus intersecting $F_0$ 
in a single boundary component.   
Indeed, we need that $X^4=D^4\cup_{N(K_0)} \Sigma\times D^2$ be homeomorphic to $D^4$, 
where $N(K_0)$ denotes a tubular neighborhood of $K_0$ within $S^3=\partial D^4$.
The fundamental group  $\pi_1(X^4)$ must be trivial and hence, by Van Kampen, 
each connected component of $K_0$ belongs to a single connected component of $\Sigma$. 
By Mayer--Vietoris, $H_1(\Sigma)\cong H_1(N(K_0))$. If one component of $\Sigma$ were 
a disk, then $F_0$ would be inessential in $F$. Thus all components of $\Sigma$ have nontrivial 
homology and hence they all must be annuli. 

Consider next the case when $\phi_0$ is not the monodromy of a fibered knot, namely that 
the compact mapping torus $M_{\phi_0}$ is not homeomorphic to 
a link complement in $S^3$. This is equivalent to the fact that Dehn surgery on 
$M_{\phi_0}$ produces a 3-manifold $M^3\neq S^3$. On the other hand the cores of the solid tori arising in the Dehn surgery process  
form a link $K_0\subset M^3$ which is fibered with monodromy $\phi_0$. 
By hypothesis there exists some surface $\Sigma$ such that the expansion of $M_{\phi_0}$ 
by $\Sigma$, namely, $E=M_{\phi_0}\cup_{N(K_0)} \Sigma\times S^1$, is homeomorphic to the 
closed complement $S^3-N(K_1)$ of some link $K_1\subset S^3$. Moreover, $S^3-N(K_1)$ 
has a fibration over $S^1$ with fiber $F_1$ which is part of an open book decomposition. Now $S^3$ bounds a disk $D^4$ which contains the trivial product $\Sigma\times D^2\subset D^4$. The embedding of $\Sigma\times D^2$ into 
$D^4$ is such that $\Sigma\times S^1\subset E$ is the standard inclusion. Therefore, 
the closure $Z^4$ of $D^4-\Sigma\times D^2$ is a 4-manifold with  boundary $M^3$. 
Now  $\pi_1(Z^4\cup_{N(K_0)} \Sigma\times D^2)=\pi_1(D^4)=1$ and by Van Kampen 
each connected component of $K$ is contained in a single connected component of $\Sigma$. 
Moreover, the normal closure of the image of $\pi_1(K)$ into $\pi_1(\Sigma)$ should be the whole group $\pi_1(\Sigma)$ and thus $\pi_1(K)\to\pi_1(\Sigma)$  is onto. Therefore 
$\Sigma$ consists of annuli. It follows that $F_1$ is homeomorphic to $F_0$ and 
actually $Z^4$ is just a smaller copy of $D^4$, contradicting our assumptions.   
\end{proof}

\begin{remark}
The assumption that $sup(\phi)$ be connected is necessary in order to have uniqueness 
of the fiber  subsurface $F_0$ with  monodromy $\phi$. Indeed, let
$F$ be a surface of genus $g\geq 3$ and $c,d\subset F$ be two simple curves which bound together 
a 2-holed torus embedded in $F$. Let $a,b,c,d$ be a pants decomposition of this 2-holed torus
into two pairs of pants $F_1$ and $F_2$.     
Then the product of the two Dehn twists along the simple curves 
$a$ and $b$ is the monodromy of a fibered link with fiber which can be either one of the 
subsurfaces $F_i$.  
\end{remark}

\begin{remark}
Mapping classes $\phi\in M(F)$ having  essential/injective 
support $sup(\phi)$  strictly smaller than $F$ cannot be torsion. Indeed 
$M(F_0)$ is torsion free, when $F_0$ has boundary. Note that there exist 
examples of holomorphic maps on complex surfaces whose associated monodromy 
is torsion, as is the case of elliptic fibrations. However, these critical points are not relatively minimal. For instance Kodaira singularities of type II, III and IV have local Milnor fibers  1-holed, 2-holed and 3-holed tori, respectively. 
The generic torus fiber $F$ is obtained by capping off the 
boundary circles by disks in each case and the monodromies 
are of finite order  $\left(\begin{array}{cc} 1 & 1 \\ -1 & 0 \end{array}\right)$, $\left(\begin{array}{cc} 0 & 1 \\ -1 & 0 \end{array}\right)$ and 
$\left(\begin{array}{cc} 0 & 1 \\ -1 & -1 \end{array}\right)$, respectively.

\end{remark}

\begin{remark}
Note that we can as well consider instead of closed manifolds $M^4$ more general 
open manifolds, as complements of 2-complexes into closed 4-manifolds and singular fibrations over 
higher genus surfaces instead of the 2-sphere. The generic fiber is then a possibly punctured 
surface $F$ which might have boundary.  Braided surfaces as considered in \cite{FP20,Kamada2} 
are then particular cases of singular fibrations whose generic fibers are punctured disks.  
\end{remark}

\begin{remark}
We considered above the topological equivalence of singular fibrations. One reason is that 
the local link $K$ determines uniquely the germ $\psi_K$ up to topological equivalence. 
Morse singularities are known to be locally smoothly equivalent, but there are germs 
of isolated singularities which are topologically equivalent while not smoothly equivalent. 
However, the topological equivalence of singular fibrations can be upgraded to smooth equivalence if we fix for each fibered knot $K$ a local model like $\psi_K$ and restrict ourselves to  
those singular fibrations which are locally smoothly equivalent to one of the given local models. 
\end{remark}

\section{Constructing singular fibrations}
\subsection{Deformations of singular fibrations}
We say that a singularity is {\em $\pm$-holomorphic} if it is locally topologically equivalent to a 
holomorphic germ with an isolated singularity. We emphasize that the local homeomorphism 
need not preserve the orientation. 
Problem \ref{charconj} above has an affirmative answer when the singularities are $\pm$-holomorphic, a result which seems folklore (see e.g. \cite{AC1,NR}):  

\begin{proposition}\label{holomorphic}
If there exists $f:M^4\to N^2$ with finitely many $\pm$-holomorphic singularities then  
$M^4$ admits an achiral Lefschetz fibration on $N^2$.
\end{proposition}

\begin{proof}
Let $p$ be a critical point of the smooth map $f:M^4\to N^2$. 
There exist some open neighborhoods $p\in U\subset M$, $f(p)\in V\subset N$ such that, 
up to left/right composition by diffeomorphisms, the restriction 
$f:U\to V$ is $\pm$-holomorphic with an isolated singularity at $p$ 
and we can assume that $U\subset \C^2$, $V\subset \C$.  
Therefore  there exists some Milnor ball $B\subset U$, such that for a small enough 
disk $D\subset f(B)\subset V$ centered at $f(p)$ the local fiber $f^{-1}(z)$
is transverse to $\partial B$ for any 
$z\in D_{\delta}$.

Moreover, Milnor has proved that for small enough $D$ and $B$ the 
tube $T=f^{-1}(D)\cap B$ is actually diffeomorphic 
to a 4-disk. A word of caution is needed, since $\partial T$ is a piecewise smooth 
3-sphere.  Actually ${\rm int}(T)$ is diffeomorphic to  ${\rm int}(D^4)$ 
by a diffeomorphism extending to a piecewise smooth homeomorphism 
between the closures.  

Let $L:\C^2\to \C$ be a linear map.  It is well-known (see \cite{HL}) that 
$f_t=f+tL:\C^2\to \C$   has only finitely many critical points in $T$ for small $t$.  
Moreover, for generic $L$ we can choose $\mu>0$ with the property that $f_t$ has 
only a finite number of nondegenerate critical points for all $0 <t\leq \mu$, we have 
$D\subset f_t(B)\subset V$ and 
for any $z\in D$ the local fiber $f_t^{-1}(z)$
is transverse to $\partial B$. In particular, we find again that 
$T_t=f_t^{-1}(D)\cap B_{\varepsilon}$ is  diffeomorphic 
to a 4-disk ambiently isotopic to $D$.

Up to composing with an ambient isotopy  sending  $T$ into $T_t$ we can assume that 
$T_t=T$. This means that $f|_{T}$ is isotopic rel boundary to a map $f_t$ 
having only quadratic singularities and $f|_{\partial T}={f_t}|_{T}$. 
Replacing $f|_T$ by ${f_t}|_T$ we obtain a continuous map with only quadratic 
singularities. 

In order to obtain a smooth function one should note that not only $T_t$ is 
isotopic to $T$ but also the link $f^{-1}(z)\cap \partial B$ is isotopic to 
$f_t^{-1}(z)\cap \partial B$ and the corresponding open book decompositions 
are isotopic. This provides an isotopy between $f|_T$ and a map which coincides  with the restriction of $f_t$ to a slightly smaller ball $T'\subset T$ and which has no critical points within $T\setminus T'$. 

Consequently we can change $f$ by a small isotopy centered around a critical point such that the new function has only quadratic singularities (of either orientation). 
\end{proof}

\begin{remark}
Note that if $f$ were holomorphic, the modified map is not necessarily 
holomorphic anymore. 
\end{remark}

\subsection{Shells}
We can refine the computation of $\varphi_s(M^4,N^2)$ by restricting the family of maps $f:M^4\to N^2$ allowed, for instance by prescribing the  
topological type of the generic fiber $F$. Note that the isotopy type of 
the embedding $F\subset M^4$ is independent of the choice of the fiber, for a 
singular fibration. 

We observed in the previous section that the topology of a singular fibration 
$f:M^4\to N^2$ is captured by a disk block $f:X^4\to D^2$ over 
some disk $D^2\subset N^2$ containing all the critical values, along with a fibration 
over $N^2\setminus {\rm int}(D^2)$. Thus singular fibrations with the same 
isotopy classes of generic fibers all come from distinct singular 
fibrations of a disk block with prescribed monodromy.

\begin{definition}
The {\em shell} of a disk block $X^4\to D^2$ is the homeomorphism type 
of the manifold $X^4$ along with its vertical 
boundary fibration  $\partial X^4\setminus {\rm int}(\partial F\times D^2)\to S^1$.
Then, two singular fibrations $X^4\to D^2$  are {\em shell equivalent} if they have isomorphic shells, namely if 
there is a homeomorphism between the disk blocks which restricts to an isomorphism 
between the vertical boundary fibrations lifting the identity.  
\end{definition} 
We can use surgery on disk blocks having the same shell, in order 
to simplify singular fibrations without changing the source and target. 
To understand  disk blocks having the same shell we need a description of the 
topology of the disk block $X^4$ along with the boundary monodromy. 
Although these two invariants are related, they do not determine each other.

\subsection{Spines}
Recall that $C(L)$ denote the cone over $L$. 
To every surface $F$ and a collection  $\mathbf \Gamma=\{\Gamma_i\}$ of graphs embedded in $F$ 
we associate the 2-polyhedron:
\[ \Pi(F, \mathbf \Gamma)= F \cup\cup_i C(\Gamma_i).\]

Let $f:X^4\to D^2$ be a disk block with generic fiber $F$ and $F_i$  the local Milnor fibres of 
its singularities. We write $A\searrow B$ if there is a triangulation of $A$ for which $A$ collapses onto $B$. 
Let $\Gamma_i\subset F_i$ be a {\em trivalent spine} of $F_i$, namely,  an embedded graph 
whose vertices are of degree 3 such that $
F_i\searrow  \Gamma_i$. To such a collection  $\mathbf \Gamma=\{\Gamma_i\}$ of graphs we associate the polyhedron:  
\[ \Pi(f)=\Pi(F,\mathbf\Gamma).\]

\begin{proposition}
If $f:X^4\to D^2$ is a disk block, then $X^4\searrow \Pi(f)$.  
\end{proposition}
\begin{proof}
Note that  a singular fiber around the singularity $p_i$ is homeomorphic to 
the cone over $\partial F_i$. Thus the singular fiber is homeomorphic to 
the quotient $F/F_i$. Let $\bigvee_i I_i$ denote the bouquet of intervals $I_i=[0,1]$ 
having the endpoints $0_i$ of each $I_i$ identified together. 

The polyhedron $\left(F\times\bigvee_i I_i\right)\cup\cup_i C(F_i)$, where $C(F_i)$ is attached to 
$F\times {1_i}$ is identified with  
the preimage $f^{-1}(\bigvee_iI_i)$, where $\bigvee_iI_i\subset D^2$ is a bouquet of 
intervals embedded in $D^2$ such that the points $1_i$ coincide with the critical values of $f$. 
Moreover, $X^4$ retracts onto the polyhedron $\left(F\times\bigvee_i I_i\right)\cup\cup_i C(F_i)$.
Note that $\Pi(F,\mathbf\Gamma)$ embeds in $X^4$ and  
\[ \left(F\times\bigvee_i I_i\right)\cup\cup_i C(F_i)\searrow F\times \{0\}\cup_i (F_i\times I_i\cup C(F_i))\cong 
F\cup\cup_i C(F_i)\searrow F\cup\cup_i C(\Gamma_i).\] 
Neumann and Rudolph (\cite{NR}, section 2, Lemma 2.1) described this collapse by pictures in the case when $f$ is the restriction of a singular fibration $D^4\to D^2$.  
\end{proof}

By using isotopies we can assume that the graphs $\Gamma_i$ have pairwise transverse intersections.  
In particular $\Gamma=\cup_i\Gamma_i\subset F$ is a graph whose vertices have valence 3 or 4.

\subsection{Shadows} 
When $\Gamma_i$ are all simple closed curves, i.e. $f$ is an achiral Lefschetz fibration, 
then $\Gamma$ is a 4-valent graph. 
In this case $\Pi(f)$ is a {\em simple} 2-polyhedron without boundary, namely every point of it has a neighborhood homeomorphic to either a plane (regular point), the union of three half-planes meeting along a common boundary line (edge point) or the cone over the 1-dimensional skeleton of a tetrahedron (vertex point). The  set of singular points of $\Pi(f)$ is the graph $\Gamma$.  
Note that the embedding of $\Pi(f)$ in $X^4$ is {\em locally flat}, namely around each point 
there exists a PL embedded $\R^3$ in which the neighborhood in $\Pi(f)$ is homeomorphic to  one of the standard embeddings of the 3-dimensional local models above.

The 2-polyhedron $\Pi(f)$ does not determine uniquely the disk block $X^4$. 
Since $X^4$ retracts onto the 2-polyhedron $\Pi(f)$, 
$X^4$ is a 2-handlebody, namely it has a handlebody decomposition with only 0, 1 and 2-handles.
Turaev defined in \cite[chap. IX]{Tu} an integer shadowed polyhedron structure on 
simple spines of 2-handlebodies. 

The connected components of  the set of regular points of a simple polyhedron $\Pi$ 
are called {\em regions}. 
The inclusion of the closure $\overline{R}$  of a region $R$ into $\Pi$ 
extends to a local homeomorphism  from a simple polyhedron $\widetilde{R}\supset \overline{R}$ into $\Pi$ whose image is a regular neighborhood of $\overline{R}$. 
The polyhedron $\widetilde{R}$ is obtained from 
$\overline{R}$ by adjoining a number of annuli and M\"obius bands to each boundary component of 
$\overline{R}$ and is uniquely determined by $\Pi$. The parity of the number of M\"obius bands required for each region $R$ is an intrinsic invariant of $\Pi$, called the $\Z/2\Z$-gleam.  The  $\Z/2\Z$-gleam is not necessarily trivial, though   $X^4$ is an orientable 
4-manifold. The presence of M\"obius bands is precisely the obstruction for the 2-polyhedron to admit a thickening to a 3-manifold, see (\cite{HAM}, p.38-39).

An integral shadow on the simple polyhedron $\Pi$ is the choice of a half-integer, called {\em gleam}, for each region of $\Pi$ lifting the canonical $\Z/2\Z$-valued gleam on $\Pi$, namely, such that the difference with  half of the $\Z/2\Z$-gleam is an integer for each region.

When $\Pi$ is a simple  locally flat polyhedron in a 4-manifold $X^4$  
there is a natural integer  gleam  $gl$ on $\Pi$ , see \cite[chap IX, section 1.6]{Tu}.
We consider the 4-dimensional neighborhood of $\widetilde{R}$ induced from $X^4$. 
Then $\widetilde{R}-R$ defines a line normal to  $\overline{R}$ along $\partial \overline{R}$
and hence a section of the projectivized normal bundle over $\overline{R}$. The obstruction 
to extending this section over  $\overline{R}$ is an element of 
$H^2(\overline{R},\partial \overline{R};\Z)$ which can be identified to $\Z$, by fixing an orientation on $R$. The value of the gleam on $R$ is then half  of this integer.

Given an integer shadowed 2-polyhedron (with empty boundary) $(\Pi, gl)$ 
Turaev defined  in \cite[chap IX, section 6]{Tu} a canonical thickening 
$\Theta(\Pi,gl)$, which is well-determined up to orientation-preserving PL homeomorphism. The canonical thickening permits to recover the disk block of an achiral Lefschetz fibration: 
\[ X^4 = \Theta(\Pi(f),gl(f))\]

The Matveev-Piergallini moves on simple 2-polyhedra  
permit to relate any two simple polyhedra having the same 3-deformation type.  Shadow equivalence 
amounts to relate integer shadowed 2-polyhedra by natural extension of these moves, as 
defined by Turaev in \cite[chap VIII, section 1.3]{Tu}.  
However, two integer shadowed simple spines of the same 2-handlebody are not necessarily shadow equivalent. The stabilization of an integer shadowed 2-polyhedron is obtained by gluing 
a number of disks  along  boundaries of disks away from the singular set, 
with trivial gleam. The canonical thickening of a stabilization corresponds to 
boundary connected sum with  copies of $S^2\times D^2$.

The following is then equivalent to \cite[Thm.1.7, Thm. 6.2, chap.IX]{Tu}:
\begin{theorem}\label{Turaev-shadows}
If the disk blocks of the achiral Lefschetz fibrations $f$ and $g$ are PL homeomorphic  
then the integer shadowed polyhedra 
$(\Pi(f),gl(f))$ and $(\Pi(g),gl(g))$ have stabilizations which are shadow equivalent. 

Further, if  the integer shadowed polyhedra $(\Pi(f),gl(f))$ and $(\Pi(g),gl(g))$ 
associated with achiral Lefschetz fibrations $f$ and $g$ are shadow equivalent, then 
the corresponding disk blocs are PL homeomorphic.  
\end{theorem}

A'Campo considered in \cite{AC} a construction of a large family of fibered links from a connected {\em divide}, namely a curve immersion into the disk. 
The monodromy of these fibered links can be recovered directly in terms of the combinatorics of the  underlying divide (see \cite{AC},  Section 3, Thm. 2). 
Ishikawa and Naoe  explained in \cite{IN} how a shadowed polyhedron is associated with any admissible divide. The corresponding 4-manifold associated with such a shadowed polyhedron admits a natural Lefschetz fibration which
is isomorphic to the Lefschetz fibration associated with the admissible divide. 
We wonder whether, more generally, two shell equivalent achiral Lefschetz fibrations $f$ and $g$ have shadow equivalent integer shadowed polyhedra 
$(\Pi(f),gl(f))$ and $(\Pi(g),gl(g))$. This is not true if we do not require 
the polyhedra to have the same vertical boundary monodromy.

\begin{remark}\label{2handlebody}
Recall that for a disk block $f:X^4\to D^2$ the source manifold  
$X^4$ is a 2-handlebody. 
On the other hand according to \cite{Harer} $X^4$ has an achiral Lefschetz fibration 
over $g:X^4\to D^2$ if and only if $X^4$ is a 2-handlebody. 
However $f$ and $g$ are not necessarily shell equivalent, since  the vertical boundary 
fibrations might differ.  
\end{remark}

\begin{remark}
The shell equivalence class of a disk block $X^4\to D^2$ containing a single critical point
is determined by the monodromy mapping class, up to conjugacy, and the embedding of the local Milnor fiber into the generic fiber. 
More generally, the shell equivalence class of a disk block having several  critical points is determined by the diagonal conjugacy class of the ordered local monodromies (depending on a fixed system of arcs) and the embeddings of the local Milnor fibers in the generic fiber. In particular, the global monodromy of the vertical boundary fibration might not be enough to determine the diffeomorphism 
type of the disk block.    
\end{remark}

\begin{remark}
To recover $X^4$ from its spine $\Pi(f)$ in general we also need the corresponding monodromies 
$\phi_i\in M(F_i)$ of the local links of singularities, subject to the 
relation $\phi_1\phi_2\cdots\phi_n=\phi$, where $\phi$ is the monodromy of the vertical 
boundary fibration.  
\end{remark}

\subsection{Unfoldings}\label{unfoldings} 
The case when the disk block $f:X^4\to D^2$ is the singular fibration associated 
with an adapted neighborhood of a cone-like singularity is particularly interesting. 
In this situation $X^4=D^4$ and the disk block is described by the (open book decomposition of the) local link $K$ of its singularity.   
In \cite{NR} Neumann and Rudolph defined an {\em unfolding} of the fibered link $K$ into the 
collection of fibered links $K_1,K_2,\ldots, K_p$ as being a 
shell equivalent singular fibration whose singular points have local links $K_1,K_2,\ldots, K_p$. 
Note that $K_j$ are actually {\em oriented} links, because 
oriented fibered links in $S^3$ are endowed with 
open book decompositions well-defined up to isotopy. 

For instance a complex algebraic link $K$ (i.e. the link of an isolated singularity of a complex polynomial map $\C^2\to \C$)  unfolds to the collection of $\mu(K)$ positive Hopf links, where $\mu(K)$ denotes the 
{\em Milnor number} of $K$, namely the first Betti number of its local 
Milnor fiber. In (\cite{NR2}, p. 110-111) the authors note that more generally closures of 
homogeneous braids are links which unfold completely into Hopf links. 
We will see later that the Hopf invariant $\lambda(K)$ of oriented fibered links $K$ 
(see \cite{Rud} and the next section) is preserved in unfoldings, namely it is the sum 
of $\lambda(K_i)$ for any unfolding of $K$ into a collection $K_i$ of links. 
Furthermore $\lambda(K)$ vanishes for all local links of isolated singularities of 
holomorphic maps. Moreover, $\lambda(K)$ is positive for the mirror images of local links of isolated singularities of holomorphic maps. Neumann and Rudolph and later Inaba (\cite{Inaba,NR3}) 
constructed fibered links occurring as  local links of real algebraic isolated singularities with arbitrarily prescribed $\lambda(K)\in\Z$.  When  
$\lambda(K)$ is negative,  the link $K$ cannot unfold to a collection of  local links of $\pm$-holomorphic maps,  in particular to (positive or negative) Hopf links.

Given a disk block $f:D^4\to D^2$, we can construct an achiral Lefschetz fibration $g:X^4\to D^2$ with the same monodromy $\phi$, by choosing a  
sequence of simple curves $\gamma_i$   and 
$\varepsilon_i\in\{-1,+1\}$ such that $\phi=\prod_i T_{\gamma_i}^{\varepsilon_i}$. 
However, $X^4$ might not be anymore diffeomorphic to $D^4$, unless the associated 2-polyhedron $\Pi(F, \Gamma)$ is collapsible 
(see \cite[Thm. 2.4]{NR}). Observe that $X^4$ is obtained from 
$F\times D^2$ by adjoining 2-handles along the curves $\gamma_1,\gamma_2,\ldots,\gamma_n$, with the framings 
$-\epsilon_1,-\epsilon_2,\ldots, -\epsilon_n$, respectively (see 
\cite{GS}, 8.2).

A collection of fibered links is an unfolding of a  
connected sum of its members. Instead of the fiber-sum above we can 
therefore use {\em Murasugi sums} of disk blocks 
$f_i:D^4\to D^2$, $i=1,2$ with collapsible spines $\Pi(f_i)$. 
Set $F=F_1\#_{D}F_2$ for the union along a disk $D$ 
of the fibers $F_i$ such that $\partial D\subset \partial F_1\cup \partial F_2$. Observe that $\Pi(f_1)\#_{D}\Pi(f_2)$ is also collapsible and hence its regular neighborhood is PL homeomorphic to $D^4$.
Therefore, when expanding the fibers $F_i$ we obtain a disk block 
$f:D^4\to D^2$ with generic fiber $F$ and monodromy 
$\iota_{F_1}(\phi_1)\iota_{F_2}(\phi_2)$, 
where $\phi_i$ are the monodromies of the $f_i$. 
This construction corresponds to the Murasugi sum of links, when $f_i$ 
are Milnor fibrations. Thus a collection of fibered links is an unfolding 
of any Murasugi sum of its constituents.  
Recall that one uses the term {\em plumbing}, when each $\partial F_i$ contains two arcs of $\partial D$.

Neumann and Rudolph (\cite{NR3}) considered the Grothendieck group of  isotopy classes of oriented fibered links, where  $K$ is the formal 
sum of $K_i$ if $K$ unfolds to the collection $K_i$.   Specifically, two links are {\em stably unfolding equivalent} if 
they become unfolding equivalent after summing with the same collection of links. The Grothendieck group  of fibered links was described by Giroux and Goodman in \cite{GG}:

\begin{theorem}\label{Giroux-Goodman}
Two collections of oriented links in $S^3$ are stably unfolding equivalent if and only if their total Milnor number $\mu$ and total Hopf invariant $\lambda$ agree.  
\end{theorem}
Actually Giroux-Goldman proved in \cite{GG} the stronger result that 
any two fibered links in $S^3$ can be obtained one from another by 
plumbing and deplumbing Hopf bands. 
As an immediate corollary we derive: 
\begin{corollary}\label{unfoldeq}
A collection of fibered links is stably unfolding equivalent to a  collection of Hopf links if and only if  $0\leq \lambda(f)\leq \mu(f)$, where $\lambda(f)$ and $\mu(f)$ denote the total Hopf invariant and the 
total Milnor number of the associated local links. 
\end{corollary}
\begin{proof}
Recall that $\lambda$ equals $1$ for the positive Hopf link and 
$0$ for the negative Hopf link (see \cite{Rud}), while both Hopf links have $\mu=1$.  
\end{proof}


This suggests that the answer to Problem \ref{charconj} should be   
negative for general 4-manifolds.

\begin{remark}
It seems unknown for which pairs $(a,b)\in \Z\oplus \Z$ there 
exists some fibered link $K$ such that $(\lambda(K),\mu(K)-\lambda(K))=(a,b)$. However, examples from (\cite{NR3}, Prop. 9.3 and p.96) of 
fibered links consisting of pairs of co-axial torus knots show that 
$\lambda(K)/\mu(K)$ is unbounded.  In particular,  if $S$ is a set of links  
such that every fibered link unfolds to a collection of links from $S$, then 
$S$ should be infinite. 
\end{remark}

\subsection{Group complexity} 
We can derive a group theoretical invariant out of this construction. 
A {\em surface presentation} of the group $G$ is its realization as a fundamental group of 
some polyhedron $\Pi(F,\Gamma)$, for some closed orientable surface $F$ and 
a collection of simple closed curves $\Gamma$. 

\begin{proposition}
Every finitely presented group $G$ admits a surface presentation. 
\end{proposition}
\begin{proof}
Gompf proved in \cite[Thm. 4.1]{Gompf} that any finitely presented group $G$ is the fundamental group 
of a closed 4-dimensional manifold $M^4$ endowed with a symplectic form, whose class in 
$H^2(M^4)$ can be assumed to be integral.  By a theorem of Donaldson (see \cite{Donald}) 
there exists a Lefschetz pencil on $M^4$ with symplectic fibers homologous to 
a multiple of the Poincar\'e dual of the class of the symplectic form. By blowing up 
the base locus we obtain a Lefschetz fibration $f$ of 
$M^4\#_{k}\overline{\mathbb C\mathbb P^2}$.  Let $X^4\to D^2$ be the disk block containing 
all singularities of $f$. Then $X^4$ retracts onto a 2-polyhedron homotopy equivalent to 
$\Pi(f)$, which is of the form $\Pi(F,\mathbf \Gamma)$, for some closed orientable surface $F$ and a collection of simple closed curves $\mathbf\Gamma$, as claimed. 
\end{proof}

We can go further and define a {\em surface-graphs presentation} of the group $G$ as its realization as a fundamental group of 
some polyhedron $\Pi(F, \Gamma)$, for some closed orientable surface $F$ and 
a collection $\Gamma$ of connected graphs. 

It is then clear that 
the number $\varphi_{s}(M^4,S^2)$ is bounded below by the minimal number of graph-relators in a surface-graphs presentation of $\pi_1(M)$.

\subsection{Singular fibrations with a single critical point}
We now aim at explaining how the classification of 
essential singular fibrations given in Proposition \ref{classification} can fail 
when nonessential singularities occur. To this purpose we consider now 
singular fibrations over $S^2$ with a single critical point. 
We first observe that not every closed oriented 4-manifold admits such a fibration: 
  
\begin{proposition}\label{unu}
\begin{enumerate}
\item 
If $M^4$ is a closed oriented 4-manifold and $\varphi_{s}(M,S^2)=1$, then $\pi_1(M)$ is 
the quotient of a free product of finitely generated free groups and surface groups by a 
subgroup isomorphic to $\Z$.

\item Let $M^4$ be a complex projective surface $M^4$ with trivial first Betti number 
whose fundamental group is neither trivial nor cyclic. If  
$M^4$ is not homeomorphic to a fiber bundle over the 2-sphere, then  
$\varphi_{s}(M^4,S^2)$ is finite and larger than 1.  
\end{enumerate}
\end{proposition}

\begin{proof}
(1) Let $X$ be the disk block containing the critical point of the singular fibration 
$M\to S^2$ with generic fiber $F$.  
If $F_1$ is the local Milnor fiber of the critical point, then 
$\Pi(f)$ is homotopy equivalent to $F/F_1$. It follows that 
$\pi_1(X)\cong \pi_1(F/F_1)$ is a free product of  finitely generated free groups and surface groups. Furthermore, $M$ is the union of $X$ and $F\times D^2$ along the 
trivial boundary fibration and hence, by Van Kampen, 
\[ \pi_1(M)=\pi_1(X)*_{\pi_1(F\times S^1)} \pi_1(F\times D^2)\cong \pi_1(X)/\pi_1(S^1).\]

(2) A complex projective surface has holomorphic maps onto $\CP^1$. These have isolated singularities and hence $\varphi_{s}(M^4,S^2)$ is finite. 

We now claim that a 4-manifold $M^4$ as in the statement cannot be a fiber bundle over the 2-sphere. Suppose the contrary and let $F$ be the fiber. We can suppose that $F$ is not a 2-sphere, as $\pi_1(M^4)$ is nontrivial.  From the long exact sequence in homotopy we derive:   
\[ 0\to \pi_2(M^4)\to \Z\to \pi_1(F)\to \pi_1(M^4)\to 0.\]
If $\pi_2(M^4)=0$, then  $\pi_1(M^4)$ is a quotient of a surface group by 
$\Z$ and hence must be trivial since its rational abelianization is trivial, which is a contradiction. 
Otherwise, $\pi_2(M^4)\subset \Z$ is isomorphic to $\Z$. As $\pi_1(F)$ is torsion-free, 
the homomorphism $\pi_2(M^4)\to \Z$ must be onto and hence 
$\pi_1(M^4)$ is isomorphic to a surface group. This contradicts our hypothesis again and 
hence our claim follows.

Eventually, assume  that $\varphi_{s}(M^4,S^2)=1$.  By the first statement of the proposition 
$\pi_1(M^4)$ is the quotient of a group $G$, which is a free product of free groups and surface groups, by a single relation. Since  $H_1(M^4;\mathbb Q)=0$, the K\"ahler manifold $M^4$ is of non-fibered type, i.e. it does not admit holomorphic maps onto a 
Riemann surface of genus at least 2. 
By \cite{Amo}, every presentation of the fundamental group 
of a non-fibered compact K\"ahler manifold with trivial first Betti number 
has positive deficiency, namely the number of relators is larger than or equal to the number of generators. This implies that $G$ is either trivial or isomorphic to $\Z$ and hence 
$\pi_1(M^4)$ is trivial or cyclic. This proves the second statement of the proposition.  
\end{proof}

\begin{remark}
Fake projective planes, also called Mumford surfaces, are well-known examples 
of complex projective surfaces with trivial first Betti number (see \cite{PY}). Their fundamental groups are arithmetic lattices in $PU(2,1)$, satisfying the conditions in Proposition \ref{unu}.  Therefore,  $\varphi_{s}(M^4,S^2)$ is finite and larger than 1, whenever $M^4$ is a fake projective plane.  
\end{remark}

We now provide a construction for singular fibrations over $S^2$ 
with a single critical point. Let $K$ be the local link of such a singularity, 
$F_0$ its Milnor fiber, $\phi\in M(F_0)$ its open book monodromy and $F$ the generic fiber 
of the singular fibration. Denote by $F_1$ and $F_2$ the result of adjoining 
to the surface $F_0$ all  disk components and all 
disk and annular components of $F\setminus {\rm int}(F_0)$, respectively. 
By language abuse we call the components of $\partial F_0$ spherical or annular if 
they belong to disks or annuli in $F\setminus {\rm int}(F_0)$. Two annular components 
form an annular pair if they bound an annulus. The remaining components will be called generic. 

In the sequel we shall assume that $F_1$ has nonempty boundary. 

A {\em braid-like} link $L$ in the product $F_1\times S^1$ is the closure 
of a braid on $F_1$, namely a link $L$ which intersects transversely every slice $F\times *$ in the same number of points.  Let $\widehat{F_1}$ denote the surface $F_1$ with all its annular boundary components capped off by 2-disks.  Given a braid-like link $L\subset F_1\times S^1$ we consider $\widehat{L}\subset \widehat{F_1}\times S^1$ 
to be the disjoint union of $L$ with $Q\times S^1$, where $Q$ are the centers of the 
disks attached to the annular points.  
Suppose that the monodromy $\widehat{\phi}$ of 
$\widehat{L}$ is the product of $\iota_{F_1}(\phi)$ with a product of Dehn twists along the annular boundary components of $\partial F_1$, 
where $\phi$ is the monodromy of $L$.

Since $F_0$ has nonempty boundary, the mapping class $\phi$ induces an automorphism of the free group $\pi_1(F_0)$, well-defined up to conjugacy. Consider $a_1,a_2,\ldots, a_N$ be 
a system of free generators for $\pi_1(F_0)$. We also denote by $b_1,b_2,\ldots,b_r$ 
the homotopy classes of based loops encircling once each one of the $r$ boundary components 
of $\partial F_0$. It is wellknown that we can choose generators such that 
$b_i=a_i$, for all $i$, if the genus of $F_0$ is positive. However, when the genus of $F_0$ is zero, we can take $b_i=a_i$, for $i\leq r-1$, while $b_r=(a_1a_2\cdots a_{r-1})^{-1}$.  
We keep this notation in order to have a unified treatment of both cases below.

Let now $n$ be the number of disk components in $F\setminus {\rm int}(F_0)$ and 
$m$ be the number of annuli. 

Then $S^3\setminus N(K)$ is identified with the mapping torus 
$M_{\phi}=F_0\times [0,1]/_{(x,0)\simeq (\phi(x),1)}$ of a homeomorphism lifting $\phi$. 
The only constraint on $\phi$ is the fact that the result of Dehn surgery 
on  $S^3\setminus N(K)$ corresponding to adding back solid tori $N(K)$ 
must be a 3-sphere. This gluing procedure is canonical. 
We glue back $S^1\times D^2$ along  a boundary torus $c\times S^1$ of 
the mapping torus $M_{\phi}$, where $c$ is a component of $\partial F$, such that 
$S^1\times p$, $p\in \partial D^2$ is sent into $c$ and 
$q\times \partial D^2$, $q\in S^1$ is sent into the {\em longitude}, i.e. the image of the arc 
$q'\times [0,1]\subset F_0\times [0,1]$ in the quotient $M_{\phi}$, where $q'\in \partial F$. 

We then denote  $\ell_s$ the (free) homotopy class of the longitude associated with 
the connected component of $\partial F_0$ indexed by $s$. 
Our aim is to compute explicitly the longitudes $\ell_s$ when 
the link $K$ arises as $\widehat{L}$, for some braid-like link in $F_1\times S^1$. 
In this case the monodromy of $K$ projects onto 
some element ${\mathbf k}\in \Z^{n+2m}$ which counts the number of left twists 
on each spherical and annular boundary component of $F_0$.  
The indices are such that $j$ and $j+m$ correspond to pairs of annular boundary components, 
when $j\geq n+1$.

\begin{proposition}\label{onecritical}
Let $M^4$ be a closed oriented 4-manifold and assume that there is a singular fibration 
$f:M^4\to S^2$ with a single critical point. Then  
\begin{enumerate}
\item the fibered link associated with the singularity is braid-like, namely 
$S^3\setminus N(K)$ is homeomorphic to the complement of a braid-like link in $F_1\times S^1$; 
\item  the monodromy $\phi$ is the product of some 
element of the framed pure braid group $FP(F_1,n)$ with  products of pairs of opposite Dehn twists 
along annular boundary components of $F_0$;    
\item Assume that $F_1$ has nonempty boundary, namely all components of $F\setminus {\rm int}(F_0)$ are not disks. Then the following  
\begin{eqnarray*} 
G(\phi)=\langle t, a_1,\ldots, a_N;\;  ta_it^{-1}=\phi(a_i), \; {\rm for}\;  1\leq i\leq N, \; \\
tb_i^{k_i}=1,  \; {\rm for}\;  1\leq i\leq n+2m, 
 t=1,\; {\rm if}\; n+2m <r\rangle
\end{eqnarray*}
is a presentation of the trivial group and moreover $k_j+k_{m+j}=0$, if $j\geq n+1$;  
\end{enumerate}
Conversely, if $\phi$ is a framed braid on $n$ strands over $F_1$ such that $G(\phi)=1$ and $k_j+k_{m+j}=0$, for $j\geq n+1$, then $\phi$ is the monodromy of a fibered link of a singular fibration with a single singularity.  
\end{proposition} 
\begin{proof}
Paris and Rolfsen determined in \cite{PR} the kernel of 
the homomorphism $M(F_0)\to M(F)$ induced by the embedding $F_0\to F$. 
More precisely,  $M(F_2)\to M(F)$ is injective while 
$\ker(M(F_1)\to M(F_2))$ is freely generated by the products of 
opposite Dehn twists along annular pairs in $\partial F_1$.  
Eventually, $K=\ker(M(F_0)\to M(F))$ is the 
so-called framed pure braid on $n$ strands on $F_1$. Let $A\subset M(F_0)$ be the 
free abelian group generated by the products of opposite Dehn twists 
along annular pairs in  $\partial F_0$. Then  
$\ker(M(F_0)\to M(F))$ is the subgroup $KA$, which  
is identified with $K\times A$, because $A$ is central. 
We supposed above that $F_1$ has boundary.

Further, according to \cite{BG} the framed pure braid group $K$ is isomorphic to 
the direct product $P(F_1,n)\times \Z^n$ of the pure braid group $P(F_1,n)$ on $n$ strands 
on $F_1$ and the free abelian group  of Dehn twists along the spherical components. 
Note that the factor $\Z^n$ corresponds to the group of Dehn twists along spherical components of $\partial F_1$. 
Hence $\ker(M(F_0)\to M(F))$ is isomorphic to the product 
$P(F_1,n)\times \Z^n\times \Z^m$. Here $\Z^m$ is the subgroup of the group 
$\Z^{2m}$ of Dehn twists along annular components of elements $\mathbf k$ with $k_j+k_{j+m}=0$, when $n+1\leq j\leq n+m$.

Consider now $\phi\in \ker(M(F_0)\to M(F))$ and $\mathbf k\in \Z^{n+2m}$ as above. 
Then the fundamental group of the mapping torus $M_{\phi}$ is the HNN extension of the free group $\pi_1(F_0)$ 
in the generators $a_i$, which admits the presentation:  
\[ \pi_1(M_{\phi})=\langle t, a_1,\ldots, a_N;\;  ta_it^{-1}=\phi(a_i), 1\leq i\leq N\rangle.\]
Further, $\pi_1(M_{\phi}\cup N(K))$ is the fundamental group of the graph of groups 
with one root vertex labeled $\pi_1(M_{\phi})$, the other vertex groups being $\Z$ and  
$N$ edges connecting the root to the vertices, all labeled $\Z\oplus \Z$. 
Therefore $\pi_1(M_{\phi}\cup N(K))$ is obtained from  $\pi_1(M_{\phi})$ by 
adjoining all relations corresponding to 
longitudes. However the class of a longitude associated to a component of $\partial F_0$ 
in $\pi_1(M_{\phi})$ reads either $tb_i^{k_i}$, $tb_j^{\pm m_j}$ or $t$, according to whether 
it is a spherical, an annular or a generic component, respectively. 
Thus $\pi_1(M_{\phi}\cup N(K))$ has the presentation:
\begin{eqnarray*} 
\langle t, a_1,\ldots, a_N;\;  ta_it^{-1}=\phi(a_i), \; {\rm for}\;  1\leq i\leq N, \; 
tb_i^{k_i}=1,  \; {\rm for}\;  1\leq i\leq n+2m, \\
\; t=1,\; {\rm if}\; n+2m <r\rangle.
\end{eqnarray*}
By Perelman's solution to the Poincar\'e conjecture, $\pi_1(M_{\phi}\cup N(K))$ is trivial if and only if $M_{\phi}\cup N(K)$ is a sphere, namely the fibration of $M_{\phi}$ over $S^1$ comes from an open book structure of a fibered link in $S^3$. 
This settles the proposition.   
\end{proof}

Computing the longitudes $\ell_i$ in $\pi_1(M_{\phi})$ seems more complicated 
when $F_1$ has no boundary, as the framed pure braid group does not split as a product
(see \cite{BG}).

\subsection{Explicit examples}\label{examples}  
We will analyze the case when $F_0$ is a 3-holed sphere. 
Then $M(F_0)=\Z^3$ is freely generated by the Dehn twists along the boundary components. 
In this case we do not need $F_1$ to have nonempty boundary. 
Let the monodromy be $\phi=T_{\gamma_1}^{k_1}T_{\gamma_2}^{k_2}T_{\gamma_3}^{k_3}$, 
where $\gamma_i$ are the boundary curves. By direct inspection the group 
$G(\phi)$ is trivial if and only if 
\[ k_1k_2+k_2k_3+k_3k_1\in\{-1,1\}.\] 
Then, up to a permutation of the $k_i$'s we have: 
\[(k_1,k_2,k_3)\in\{(\pm 1, \pm 1, 0), (2,-1,3), (-2,1,-3), (1,-1,n); \; n\in \Z\}.\]
Every integral solution  $(k_1,k_2,k_3)$ provides a fibered link $K$,
with prescribed monodromy.  We retrieve the description of 3-component fibered links of genus zero obtained by Rogers (see \cite{Rog}). 

In order to expand the local Milnor fibration we first can adjoin 
2-disks on all boundary components of $F_0$, so that the associated generic fiber is a 
$F=S^2$. For those integral solutions where two coordinates have opposite values, we can 
also adjoin an annulus and further one 2-disk along the remaining component $F_0$ so that the generic fiber is a torus.  These solutions correspond to: 
\[(k_1,k_2,k_3)\in\{(1,-1,n); \, n\in \Z\}.\] 
Eventually, if $n=0$, we could also replace the 2-disk by any surface with boundary. 
Note that the corresponding link is the pretzel link 
$(2,-2,2n)$ with the natural orientation induced from its Seifert surface.
Expanding the local Milnor fiber to $F$ yields  a disk block 
$X\to D^2$ whose boundary fibration is trivial. 
The corresponding closed manifold $M^4$ is obtained by gluing together 
this disk block with $F\times D^2$ fiberwise. 

\begin{remark}
Specifying a fibered link alone  (without the orientation) is not enough for describing the singularity, even in the case of pretzel links. For instance, the pretzel link $(2,-2,2)$ has an open book decomposition  which is different  from that given above. As the closure of a full twist on 3 strands, which is a homogeneous braid, we can use Stalling's construction to describe an open book decomposition with a 3-holed torus as fiber. However, the orientations induced 
by the two fibrations above on the pretzel link are different.     When capping off the boundary components of the 3-holed torus by disks we obtain a torus singular fibration corresponding to a type IV singularity in the Kodaira classification.      
\end{remark}

\subsection{Singular fibrations of $S^4$}\label{matsum}
Recall that $\varphi(M,N)$ denotes the minimal number of critical points
of a smooth map between the manifolds $M$ and $N$. 
To obtain upper bounds for $\varphi$  we need specific constructions of 
maps with finitely many critical points. 
Note that there is a  smooth map  $S^4\to S^2$ with 2 critical points. 
Let $H:S^4\to S^3$ be the suspension of the Hopf fibration and
$h:S^3\to S^2$ be the Hopf fibration. The composition
$h\circ H:S^4\to S^2$ is the Matsumoto singular fibration by tori (see \cite{Mat0,Mat}).
According to the position of the two critical values of $H$ with respect to
the fibers of $h$, we have two cases:
\begin{enumerate}
\item critical values belong to the same fiber; then $ h\circ H$ has one critical value;
\item critical values belong to distinct fibers and thus  $ h\circ H$ has two critical values.
\end{enumerate}
In both cases we have a smooth map with 2 critical points.
The generic fiber is a torus and the critical fibers correspond to
\begin{enumerate}
\item the union of two 2-spheres intersecting transversely
at two points with opposite orientation, called 
the Montesinos twin singularity in \cite{Mat}; 
\item two immersed spheres with transversal intersections, with one having
 self-intersection +1 and the other one -1.
\end{enumerate}

It follows that $h\circ H$ is an achiral Lefschetz fibration of $S^4$ over $S^2$. 

Note that the $\varphi(M^4, S^2)$ can be strictly smaller than the 
number of critical points in an achiral Lefschetz fibration. The simplest
example is the following one, which is due to Matsumoto (\cite{Mat2}): 
\begin{proposition}\label{matsumoto}
We have $\varphi(S^4,S^2)=1$, while the number of critical points in an achiral Lefschetz 
fibration of $S^4$ over $S^2$ is $2$. 
\end{proposition}

\begin{proof}
If $S^4$ had an achiral Lefschetz fibration with at most one critical point, 
then up to a change of orientation on $S^4$ we may suppose that the 
critical fiber has positive self-intersection. Thus $S^4$ suitably oriented would have a 
Lefschetz fibration over $S^2$, which is well-known to be false. For instance, 
a Lefschetz fibration induces an almost-complex structure (see \cite{GS}, Ex. 8.1.6) and in particular 
$b_1-b_2^+$ should be odd, by the Noether formula for almost-complex structures (see \cite{GS}, Thm. 1.4.13) contradicting the vanishing of the first two Betti numbers of $S^4$.

The inequality $\varphi(S^4,S^2)\leq 1$  follows from a
construction due to Matsumoto (\cite{Mat2}) of 
generalized torus fibrations $S^4\to S^2$ with a single singular point. 
Note that locally this map is equivalent to the germ of the isolated singularity 
$(\R^4,0)\to (\R^2,0)$ expressed in complex coordinates by the mixed polynomial: 
\[ f(z_1,z_2)=z_1z_2(\overline{z}_1+\overline{z}_2).\]
The local Milnor fiber is a pair of pants $F_0=\Sigma_{0,3}$  and its link $L$ is the pretzel link 
$(2,-2,2)$. If the link components are respectively denoted $C_0,K_0$ and $K_1$, then the 
monodromy of the corresponding open book structure is written in $M(\Sigma_{0,3})$ as the 
product of the Dehn twists:  
\[ T_{C_0}T_{K_0}^{-1}T_{K_1}.\]
We consider the expansion of $F_0\subset T$ to a torus $T$ by adjoining the 
surface $Z=D^2\cup A$, where $A= S^1\times [0,1]$, in such a way that the component $K_1$ is capped off by a disk and the components $C_0$ and $K_0$ by an annulus $A$. One obtains then a 4-manifold with boundary 
$X=D^4\cup_{N(L)} Z\times D^2$ endowed with a smooth map $f:X\to D^2$. 
Note that the monodromy of the boundary fibration $f|_{\partial X}: \partial X\to S^1$ is trivial 
and hence the boundary fibration extends over  $D^2$  providing a singular fibration 
$X\cup_{S^1\times T} D^2\times T\to S^2$ with a single critical point. 
It remains to see that $M(f)=X\cup_{S^1\times T} D^2\times F$ is diffeomorphic to $S^4$. 

To this purpose note that $S^4$ is obtained by adding to $D^4$ first a 2-handle over an unknot 
with zero framing a cancelling 3-handle and then a 4-handle. 
After adding the 2-handle we obtain $D^2\times S^2$.  
The union of the complementary 3-handle and the 4-handle is diffeomorphic to $D^3\times S^1$.

Let $F_1$ be the result of expanding $F_0$ to an annulus by adjoining  $D^2$ 
along the component $K_1$ and $Y=D^4\cup_{N(K_1)} (F_1-F_0)\times D^2$. Then 
$Y$ is the result of adding a 2-handle along the $K_1$ with zero framing. 
Remark that $K_1$ is an unknot in $S^3$ and hence $Y$ is diffeomorphic to  $D^2\times S^2$. 
Moreover, the monodromy of the boundary fibration
$\partial Y\setminus N(K_0\cup C_0)\to S^1$  
by annuli is trivial.  On the other hand we obtain  $M(f)$ from $Y$ by 
gluing back first  $Z-{\rm int}(Y)=D^2\times A$ which is a trivial fibration in annuli over $D^2$ and 
second the torus product $D^2\times T$. The union of these two pieces is
\[ D^2\times T^2\cup_{A\times \partial D^2} A\times D^2=
(D^2\times S^1\cup_{\partial D^2\times [0,1]} [0,1]\times D^2)\times S^1.\]
Note that  $[0,1]\times D^2$ is half of a solid torus $S^1\times D^2$, which 
is glued together to the solid torus $D^2\times S^1$ by identifying meridians to longitudes. 
Thus 
\[ D^2\times S^1\cup_{\partial D^2\times [0,1]} [0,1]\times D^2=S^3\setminus [0,1]\times D^2=D^3.\] 
It follows that $M(f)$ is obtained from  $S^2\times D^2$ by adjoining $D^3\times S^1$ and 
thus it is diffeomorphic to  $S^4$, as claimed.

The opposite inequality $\varphi(S^4,S^2) \geq 1$ is a consequence of the 
well-known fact that there is no Serre fibration 
$f:S^4\to S^2$ (see e.g. \cite{AndFun1}). 
\end{proof}

\section{Homotopical obstructions to singular fibring}
\subsection{The index of a 2-plane field after Hirzebruch, Hopf and Rudolph}
We have an action of $S^3\times S^3$ on $\R^4$ given by 
\[ (q_1,q_2)(x)=q_1\cdot x\cdot q_2^{-1}, \; x\in \R^4, \, (q_1,q_2)\in S^3\times S^3, \]
where $\R^4$ is identified with the field of quaternions, $S^3$ 
with the subgroup of unit quaternions and $\cdot$ denotes the multiplication of quaternions. 
This provides an exact sequence: 
\[ 1\to \Z/2\Z=\langle (-1,-1)\rangle \to S^3\times S^3\stackrel{Q}{\to} SO(4)\to 1\]
identifying  $S^3\times S^3$ with the universal covering of $SO(4)$, namely the ${\rm Spin}(4)$ group. 
If $S^1\subset S^3$ denotes the subgroup of unit complex numbers, then 
\[ Q(S^1\times S^1)=SO(2)\times SO(2)\subset SO(4).\]
Now $\R^4$ with the complex structure induced by the multiplication by $i$ is identified 
with $\C^2$ and the unitary matrices form a subgroup $U(2)\subset SO(4)$. 
The above description yields $Q(S^1\times S^3)=U(2)$. On the other hand the image 
$Q(S^3\times S^1)=U(2)'$ is another copy of $U(2)$ embedded in $SO(4)$. 
Obviously $U(2)\cap U(2)'=SO(2)\times SO(2)$.

The quotient of the $S^1$-action on  $S^3$ by left multiplication of quaternions 
coincides with the sphere $S^2$, as the base 
space of the Hopf fibration. It follows that 
\[ SO(4)/U(2)=Q(S^3\times S^3)/Q(S^1\times S^3)=S^3/S^1\cong S^2\]
and in a similar way
\[ SO(4)/U(2)'=Q(S^3\times S^3)/Q(S^3\times S^1)=S^3/S^1\cong S^2.\]

The Grassmann variety $\tilde G_2(\R^4)$ of oriented 2-planes in $\R^4$ is 
$SO(4)/SO(2)\times SO(2)$. By direct inspection the natural map 
\[ \tilde G_2(\R^4)\to SO(4)/U(2)\times SO(4)/U(2)'\cong S^2\times S^2\]
is a canonical diffeomorphism. Fixing an orientation on $S^3$ we obtain a 
canonical isomorphism $\pi_3(S^3)\to \Z$, given by the degree. Using the 
Hopf map derived from the left multiplication above we derive a canonical isomorphism 
$\pi_3(S^2)\to \Z$. Then the diffeomorphism above induces a canonical isomorphism 
\[\pi_3(\tilde G_2(\R^4))\to \Z\oplus \Z.\]

Consider now a smooth oriented 4-manifold $M$ and let $\xi$ be a  2-plane field over $M$ with 
finitely many isolated singularities at $P=\{p_1,p_2,\cdots,p_n\}$. 
The restriction of $\xi$ to a small sphere $S^3$  around a singularity $p\in P$ 
yields a map $S^3\to  \tilde G_2(\R^4)$. We denote by $(-\lambda_p,\rho_p)\in \Z\oplus \Z$ the class in 
$\pi_3(\tilde G_2(\R^4))$ of this restriction and call it the {\em index} of the 2-plane field at the singularity $p$ (see \cite{Thomas}).   
Moreover, the  (global) index $(-\lambda, \rho)$ of $\xi$ is the sum of the indices of all singularities. 

The main result of Hirzebruch and Hopf from \cite{HH} expresses the index of 
a 2-plane field in terms of homotopy invariants of the manifold. 
To this purpose we need to introduce more notation. Let $S$ be the intersection pairing 
defined on the quotient $H=H^2(M;\Z)/{\rm Tor}$ by its torsion subgroup. Recall that 
$w\in H$ is {\em characteristic} if $S(w,x)\equiv S(x,x) ({\rm mod}\; 2)$, for every $x\in H$.
Set $W$ be the set of characteristic elements of $H$ with respect to $S$ and 
\[ \Omega(M)=\{ S(w,w); w\in W\}\subset \Z.\]
When $H=0$ we put $\Omega(M)=0$. Note that 
\[ \Omega(M)\subseteq \sigma(M)+8\Z,\]
where $\sigma(M)$ denotes the signature of the intersection pairing $S$.
Let $e(M)$ denote the Euler characteristic of the manifold $M$.  
We can now state: 

\begin{theorem}[\cite{HH}, Satz 4.3]\label{Hirzebruch-Hopf}
The pair $(-\lambda,\rho)\in \Z\oplus \Z$ is the index of an oriented 2-plane field 
with isolated singularities on the closed oriented manifold $M^4$ if and only if 
there exist $\alpha,\beta\in \Omega(M)$ such that 
\[ 4\lambda=2 e(M)+3\sigma(M) -\alpha, \, 4\rho= 2e(M) -3\sigma(M)+\beta.\]
\end{theorem}

Consider now a singular fibration $f:M^4\to N^2$ over some closed orientable surface $N$. 
We can associate to $f$ the oriented 2-plane field $\xi=\ker Df$. 
The index of this 2-plane field at a critical point $p$ of $f$ coincides then with 
the pair $(-\lambda(f)_p, \rho(f)_p)$, where $\lambda(f)_p, \rho(f)_p$ are the invariants associated by Rudolph in \cite[Definition 1.4]{Rud} and called enhanced Milnor numbers there. Note the change in the sign of the first component $\lambda$, with respect to the convention in \cite{HH} needed in order to follow \cite{Rud}. 

Furthermore let us denote by $\mu(f)_p$ the Milnor number of the local germ of $f$ at $p$, namely the 
first Betti number of the local Milnor fiber. We will use next the following result of 
Rudolph: 

\begin{theorem}[\cite{Rud}, Corollary 2.6]\label{Rudolph}
If the germ $f:\R^4\to \R^2$ has an isolated singularity at $p$, then we have: 
\[ \mu(f)_p=\lambda(f)_p+\rho(f)_p.\]
\end{theorem}
 
We denote by $\mu(f)$ the sum of local Milnor numbers $\mu(f)_p$ over all singularities 
$p$ of $f$.  
 
 \subsection{Almost complex structures}
Consider an oriented 2-plane field $\tau$ over the smooth oriented 4-manifold $M^4$
having only finitely many singularities $P$. 
Then there exist two almost complex structures on $M^*=M-P$ 
canonically associated with $\tau$, which we call 
$J$ and $J'$ such that $J$ is compatible with the given orientation of $M$ and $J'$ is 
compatible with the opposite orientation of $M$. 
Assume that $M$ is given a Riemannian metric. 
Then the tangent bundle $TM$ splits  as $\tau\oplus \tau^{\perp}$ and hence 
the structure group reduces from $SO(4)$ to $SO(2)\times SO(2)$. 
Note that conversely, if the tangent bundle has a $SO(2)\times SO(2)$ bundle structure 
then it has an associated  oriented 2-plane field. 
Now, the almost complex structure $J$ is the one obtained when both $\tau$ and $\tau^{\perp}$ 
are identified with complex line bundles over $M$. Moreover $J'$ is the almost complex structure 
where $\tau$ is a complex line bundle over $M$ and $\tau^{\perp}$ is the conjugate of the previous complex line bundle.

Observe that $SO(4)/U(2)$ is the space of almost complex structures compatible with 
the orientation on $\R^4$, while  $SO(4)/U(2)'$ is the space of almost complex structures compatible with the opposite orientation on $\R^4$. 

Furthermore, consider a germ $f:\R^4\to \R^2$  with an isolated singularity at $p$. 
The associated oriented 2-plane field $\xi=\ker Df$ provides then two almost complex structures 
$J(f)$ and $J'(f)$ in $U-p$, where $U$  is an open neighborhood of the singular point $p$
not containing other singularities. The following seems to be well-known:

\begin{lemma}
Let $f:\R^4\to \R^2$ be a smooth germ with an isolated singularity at $p$. 
Then $\lambda(f)_p$ vanishes if and only if the induced almost complex structure 
$J(f)$ extends over the singularity $p$. 
Likewise, $\rho(f)_p$ vanishes if and only if the induced almost complex structure 
$J'(f)$ extends over the singularity $p$. 
\end{lemma}
\begin{proof}
Note that $\lambda(f)_p$ is the homotopy class of the map 
$S^3\to SO(4)/U(2)$ which associates to a point $x$ of the 3-sphere around $p$ 
the class of $J(f)$ at the point $x$. This map is null-homotopic if and only if 
it extends to the 4-disk, namely it extends over $p$. The proof is similar for 
$\rho(f)_p$. 
\end{proof}

Let $f:M^4\to N^2$ be a singular fibration of a closed oriented 4-manifold $M$ over a closed oriented surface $N^2$. 
Then we have two almost complex structures $J(f)$ and $J'(f)$ on $M^*=M-P$, where 
$P$ is the set of singular points of $f$.  
Let then $c_1(J(f))\in H^2(M^*)$ and $c_1(J'(f))\in H^2(M^*)$ be the first Chern classes 
of these two almost complex structures.  We denote by the same letters 
their inverse images in $H^2(M)$ under the isomorphism induced by the inclusion $H^2(M)\to H^2(M^*)$. 
Eventually let $c_1(J(f))^2, c_1(J'(f))^2\in \Z$ be the images of their squares 
into $\Z$ by the isomorphism $H^4(M)\cong \Z$.

\begin{lemma}\label{classes}
If  $(-\lambda,\rho)\in \Z\oplus \Z$ is the index of the oriented 2-plane field 
with isolated singularities associated to the singular fibration $f$, then  $c_1(J(f))$ and  $c_1(J'(f))$ are characteristic classes in $H$ and 
\[ c_1(J(f))^2=2 e(M)+3\sigma(M) - 4\lambda, \, c_1(J'(f))^2= 4\rho- 2e(M) +3\sigma(M)\]
\end{lemma}
\begin{proof}
This follows from the description of the integers $\alpha,\beta$ from theorem \ref{Hirzebruch-Hopf}, 
see \cite[Satz 3.3]{HH}. 
\end{proof}
 
We say that an isolated singularity $p$ of the smooth map $f:\R^4\to \R^2$ is {\em almost complex} or  {\em opposite almost complex} if $\lambda(f)_p=0$ and $\rho(f)_p=0$, respectively. 
This terminology extends to the fibered links occurring as local links of the 
corresponding isolated singularities. 
It was shown in \cite{Rud2} that the closure of a strictly positive braid is 
an almost complex link. In particular, this is so for all local links of complex 
holomorphic isolated singularities. 
Furthermore,  Rudolph proved  in \cite{Rud} that: 
\[ \lambda(\overline{K})=\rho(K), \, \rho(\overline{K})=\lambda(K)\]
where $\lambda(K), \rho(K)$ stands for $\lambda(\phi_K)_0, \rho(\phi_K)_0$ and 
$\overline{K}$ denotes the mirror image of $K$. In particular mirror images of 
almost complex links are opposite almost complex. 
Moreover, if the oriented fibered link $K$ is amphichiral, then 
$\lambda(K)=\rho(K)=\mu(f)/2$ and hence it is neither almost complex nor opposite almost complex, if it has nonzero genus. 

\begin{remark}
Assume that the homology class of the generic fiber $F$ of a singular fibration 
$f:M^4\to S^2$ is nonzero in $H_2(M^4)$.  
Then, there exists a symplectic structure on the open manifold obtained from $M^4$ by removing the fibers with $\lambda(f)_p\neq 0$, assuming at least one exist, by a theorem of Gromov (see \cite[section 4.1]{MS}).   
We don't know under which conditions the symplectic structure 
extends over the singular fiber $f^{-1}(p)$. 
\end{remark}

\subsection{The index of an oriented 2-plane over a 3-manifold}
Let $X^4$ be a compact oriented 4-manifold with boundary $\partial X^4=N^3$. 
Let us set $X^*=X^4\setminus \{p\}$, where $p$ is a point in the interior of $X^4$. 
If $D^4\subset X^4$ is a small 4-disk embedded in $X^4$, centered at $p$, 
let $M^4=X^4\setminus {\rm int}(D^4)$, which has boundary 
$\partial M^4=\partial X^4\sqcup \partial D^4=N^3\sqcup S^3$. 

Let $\xi$ be an oriented  2-plane field on the boundary $\partial X^4=N^3$ which is tangent to $X^4$. 
Suppose that there exists an extension $\tilde\xi$ of $\xi$ to an oriented tangent 
2-plane field on $X^4$ which has only  one singularity at the given point $p$.  
Let $\xi^{\circ}$ denote the restriction of $\tilde{\xi}$ to the 
sphere $S^3$ around $p$. 

The question addressed here is to what extent the homotopy type of 
the oriented 2-plane field $\xi^{\circ}$ is well-defined and independent on the choice of the extension $\tilde\xi$? Note that there is a canonical trivialization induced from $D^4\subset X^4$ on $S^3$. An oriented 2-plane field on $S^3$ with this trivialization corresponds to a map $S^3\to \tilde G_2(\R^4)$ and its homotopy class 
is described by the index $(-\lambda_p(\xi^{\circ}), \rho_p(\xi^{\circ}))\in \Z\oplus \Z$.

From now on we assume that $X^4$ has a spin structure. 
Then $X^4$  is almost parallelizable, namely 
$X^*$ is parallelizable. Fix once and for all a trivialization of the tangent bundle 
$TX^*$ which induces a trivialization of the tangent bundle on 
$M^4=X^4\setminus {\rm int}(D^4)$.  

Then a homotopy class of an oriented 2-plane field $\xi\subset TX^4|_{N^3}$ on $N^3$ corresponds to a couple of elements of the  second cohomotopy set: 
\[ [N^3,  \tilde G_2(\R^4)]\cong [N^3, S^2\times S^2]\cong [N^3, S^2]\times [N^3, S^2].\]

The second cohomotopy set $[K,S^2]$ of a 3-dimensional complex 
was described by Pontryagin in \cite{Pon}. First,  there is a natural map 
\[ \mu_{K}:[K,S^2]\to H^2(K;\Z),\]
which associates to the homotopy class of a map $f:K\to S^2$ the 
pull-back class $f^*[S^2]\in H^2(K;\Z)$. Moreover, Pontryagin proved 
that $\mu_{K}$ is surjective. 

Furthermore, for each $\beta\in H^2(K;\Z)$ there is a bijection between 
the homotopy classes in $\mu^{-1}(\beta)$ and  the 
set $H^3(K;\Z)/\psi_{\beta}(H^1(K;\Z))$, where 
$\psi_{\beta}(\alpha)=2\alpha\cup \beta\in H^3(K;\Z)$, for $\alpha\in H^1(K;\Z)$.
Note that $H^3(K;\Z)/\psi_{\beta}(H^1(K;\Z))$ is naturally an affine space. Namely, 
if two homotopy classes $f,g:K\to S^2$ have the same characteristic 
class $\mu_K(f)=\mu_K(g)=\beta$, then they differ by some element 
$d(f,g)\in H^3(K;\Z)/\psi_{\beta}(H^1(K;\Z))$. 

This description holds when we replace $K$ by $M^4$, $N^3$ or $S^3$. 
Note that the quotient $H^3(N^3;\Z)/\psi_{\beta}(H^1(N^3;\Z))$ has a simple description, when 
$N^3$ is an oriented 3-manifold. 
Consider some class $\beta\in H^2(N^3;\Z)$. If $\beta$ is torsion, then we set $d_{\beta}=0$. 
Otherwise, let $d_{\beta}\in \Z$ be twice the divisibility of the class 
$\beta$, namely the largest $k$ such that $2\beta$ is $k$ times a class 
in  $H^2(N^3;\Z)$ modulo torsion. Then an orientation on $N^3$ provides a canonical 
isomorphism (compare with \cite[Prop. 4.1]{Gompf2}):  
\[ H^3(N^3;\Z)/\psi_{\beta}(H^1(N^3;\Z))=\Z/d_{\beta}\Z.\]

If $\xi$ is the oriented 2-plane field above, we denote by 
$\pmb{\mu}_{N^3}(\xi)\in H^2(N^3;\Z)\oplus H^2(N^3;\Z)$ the pair of values of 
$\mu_{N^3}$ on the corresponding classes in $[N^3,S^2]\times [N^3, S^2]$. 
Given two oriented 2-plane fields $\xi$ and $\xi'$ with  
\[ \pmb{\mu}_{N^3}(\xi) = \pmb{\mu}_{N^3}(\xi')= (\beta_1,\beta_2)\in H^2(N^3;\Z)\oplus H^2(N^3;\Z),\]
there is a relative invariant: 
\[ d(\xi,\xi')\in \Z/d_{\beta_1}\Z\times  \Z/d_{\beta_2}\Z, \] 
which vanishes if and only if $\xi$ and $\xi'$ are homotopic oriented 2-plane fields.

\begin{proposition}\label{milnormod}
Assume  that $X^4$ is spin, $\tilde\xi$ an extension of $\xi$ over $M^4$ and 
$\pmb{\mu}_{N^3}(\xi)=(\beta_1,\beta_2)\in H^2(N^3;\Z)\times H^2(N^3;\Z)$. Then 
the index $(-\lambda_p(\xi^{\circ}), \rho_p(\xi^{\circ}))\in \Z\oplus \Z$  has a well-defined image in the quotient
\[  \Z/d_{\beta_1}\Z\times  \Z/d_{\beta_2}\Z,\]
independent of the choice of the extension $\tilde\xi$ on $X^*$. 
\end{proposition}
\begin{proof}
The homotopy class of a 2-plane field $\tilde\xi$ on $M^4$ corresponds to an 
element of the set 
\[ [M^4,  \tilde G_2(\R^4)]\cong [M^4, S^2\times S^2]\cong [M^4, S^2]\times [M^4, S^2].\]
The restriction to boundary components induces two maps 
\[ F: [M^4,  \tilde G_2(\R^4)] \to [N^3,  \tilde G_2(\R^4)], \;  F(f)=[f|_{N^3}],\]
\[ P: [M^4,  \tilde G_2(\R^4)] \to [S^3,  \tilde G_2(\R^4)], \;  F(f)=[f|_{S^3}].\]
Our question can be rephrased as follows. Is 
$P(\tilde\xi)$ well-defined in $[S^3,  \tilde G_2(\R^4)]$ and independent of the choice of a lift $\tilde\xi\in F^{-1}(\xi)$?

Since $\tilde G_2(\R^4)\cong S^2\times S^2$ and the homotopy set functor behaves 
well with respect to products at the target, we can look at the factors independently 
and ask the same question when $\tilde G_2(\R^4)$ is replaced by  $S^2$. 

Pick up some $\xi\in [N^3,S^2]$ with $\mu_{N^3}(\xi)=\beta\in H^2(N^3;\Z)$ and some lift 
$\tilde\beta\in H^2(M^4;\Z)$ of $\beta$ under the inclusion induced homomorphism 
$H^2(M^4;\Z)\to H^2(N^3;\Z)$. Let $\tilde\xi\in [M^4,S^2]$ be an extension such that  
$\mu_{M^4}(\tilde\xi)=\tilde\beta\in H^2(M^4;\Z)$ and denote by 
$\xi^{\circ}$ the restriction of $\tilde\xi$ to $S^3$. 
We note first that $\mu_{S^3}(\xi^{\circ})=0$ since $H^2(S^3;\Z)=0$.

Now,  $\mu_{S^3}^{-1}(0)$ admits a natural (i.e. not affine) isomorphism to $H^3(S^3;\Z)/\psi_0(H^1(S^3;\Z))=H^3(S^3;\Z)=\Z$. Indeed, we can consider the relative invariant $d(\xi)=d(\xi, \mathbf 0)$ with respect to the trivial 
2-plane field $\mathbf 0$, namely the one corresponding to a null-homotopic map $S^3\to S^2$.

Consider the map induced by the inclusion: 
\[ i^*: H^3(M^4;\Z)\to H^3(\partial M^4;\Z)\cong H^3(N^3;\Z)\oplus H^3(S^3;\Z).\]
The long exact sequence in cohomology associated to the pair $(M^4,\partial M^4)$ reads:
\[  H^3(M^4;\Z)\to H^3(\partial M^4;\Z)\stackrel{\delta}{\to} H^4(M^4,\partial M^4;\Z)\to H^4(M^4)=0.\]
Therefore the image of $i^*$ is the kernel of the boundary homomorphism $\delta$.  
Now, by standard arguments $H^4(M^4,\partial M^4;\Z)\cong \Z$ and  $H^3(\partial M^4;\Z)\cong \Z\oplus \Z$  
is freely generated by  the fundamental classes
of $N^3$ and $S^3$ with the inherited orientations from $X^4$.  
Eventually, the boundary homomorphism $\delta$ reads: 
\[ \delta(a[N^3]+b[S^3])=a-b\in \Z.\]
This implies that the image of $i^*$ coincides with the diagonal within  
$\Z\oplus \Z=H^3(\partial M^4;\Z)$, and hence the homomorphisms induced by the inclusions 
\[  H^3(M^4;\Z)\to H^3(N^3;\Z), \;  H^3(M^4;\Z)\to H^3(S^3;\Z) \]
are surjective.

For each $\tilde\beta\in H^2(X^4)$ we choose some $\xi_{\tilde\beta}$ having the properties that $\mu_{M^4}(\xi_{\tilde\beta})=\tilde\beta$ and 
$\tilde\beta|_{S^3}=\mathbf 0$. The existence of such  a $\xi_{\tilde\beta}\in [X^4,S^2]$ follows from the surjectivity of the map induced by the inclusion $H^3(M^4)\to H^3(S^3)$.

Let  
$d(\xi,\xi_{\tilde\beta}|_{N^3})=\alpha\in H^3(N^3;\Z)/\psi_{\beta}(H^1(N^3;\Z))$. 
Set $\tilde\alpha\in  H^3(M^4;\Z)/\psi_{\tilde\beta}(H^1(M^4;\Z))$ for  
the lift $d(\tilde\xi,\xi_{\tilde\beta})$ of $\alpha$ to $M^4$. 
Then  $d(\xi^{\circ})=d(\xi^{\circ}, \mathbf 0)\in \mu^{-1}(0)$ is the image of $\tilde\alpha$ into $H^3(S^3;\Z)/\psi_0(H^1(S^3;\Z))$.

In particular, if we fix the class $\alpha\in H^3(N^3;\Z)$, then the image of its 
lift $\tilde\alpha\in H^3(M^4;\Z)$ into $H^3(S^3;\Z)$ is well-defined independently on 
the lift, because the image of $i^*$ coincides with the diagonal within  
$\Z\oplus \Z$. 

Therefore, when $\mu(\xi)$ is torsion in $H^2(N^3;\Z)$, then the 
element $\xi^{\circ}$ is well-defined, up to homotopy, and independent of the choice 
of the extension $\tilde\xi$. 

On the other hand, if $d(\xi,\xi_{\tilde\beta})$ is only defined modulo $d_{\beta}$, then 
$d(\xi^{\circ})$ is only defined  modulo the image of $d_{\beta}\Z\subseteq \Z\cong H^3(N^3;\Z)$ into $H^3(S^3;\Z)$. By the discussion above its image 
is a copy of $d_{\beta}\Z\subseteq \Z=H^3(S^3;\Z)$. This settles the claim.  
\end{proof}

\subsection{Shell invariants} 
If $X^4$ is a disk block, we can take for $N^3$ any 3-manifold bounding some 
4-manifold containing all critical points. 
Note that $X^4$ is a 2-handlebody and hence it is a spin manifold. 
Moreover, it is known that any oriented 2-plane field $\xi$ on $\partial X^4$ extends 
to an oriented 2-plane $\tilde\xi$ on $X^4$ with a single singularity (see e.g. 
\cite{Gompf2}, proof of Lemma 4.4). 
Every singular fibration $f:X^4\to D^2$ induces an oriented 2-plane field 
$\tilde \xi =\ker Df$ with finitely many singularities. Let $D^4\subset X^4$ be an embedded disk containing the singularities. Observe that $\tilde \xi|_{X^4\setminus{\rm int}(D^4)}$ can be extended over $D^4$ with a single singularity $p$ in order to fit 
the hypothesis above. Moreover, the index of $\tilde \xi|_{S^3}$ 
is the total index of $\tilde\xi$.  From Proposition \ref{milnormod} we derive:  

\begin{corollary}\label{shellindex1}
Let $\xi$ be the restriction of $\ker Df$ to $N^3$ and $\pmb{\mu}_{N^3}(\xi)=(\beta_1,\beta_2)$. 
Then 
\[ (-\lambda(\xi),\rho(\xi))=(-\lambda_p(\xi^{\circ}), \rho_p(\xi^{\circ}))\in  \Z/d_{\beta_1}\Z\times  \Z/d_{\beta_2}\Z,\]
is an invariant of the shell equivalence of $X^4$.  
\end{corollary}

An interesting particular case is when $N^3=\partial X^4$ and the generic fiber is closed. 
Then the  oriented 2-plane field $\xi\subset TM^4|_{N^3}$ is actually tangent to $N^3$. 
The homotopy classes of oriented  2-planes tangent to $N^3$, also called combings, 
correspond to elements of the set: 
\[ [N^3, \tilde G_2(\R^3)]\cong [N^3, S^2].\]
The boundary is collared and we have a splitting of $TM^4|_{N^3}$ as the sum 
of $TN^3$ and a trivial line bundle, which will be fixed along with the trivialization of $TM^4$.   
This provides a well-defined  inclusion $\R^3\to \R^4$ between tangent spaces  and hence 
a map $\tilde G_2(\R^3)\to \tilde G_2(\R^4)$ which induces an injection: 
\[ \iota_{N^3}: [N^3, \tilde G_2(\R^3)]\to [N^3,  \tilde G_2(\R^4)].\]
Thus the homotopy class of a 2-plane field $\xi$ as above belongs to the image of $\iota$. 

Recall that $\tilde G_2(\R^3)=SO(3)/1\times SO(2)$ is diffeomorphic to $S^2$. 
In order to retrieve $\iota$ it is enough to describe the inclusion 
$\iota: \tilde G_2(\R^3)\cong S^2\to \tilde G_2(\R^4)\cong S^2\times S^2$.

The vector space of purely imaginary quaternions (i.e. with vanishing real part) identifies with $\R^3\subset \R^4$. Then the group of unit quaternions $S^3$ acts 
on $\R^3$ by means of the formula: 
\[ q(y)=q\cdot y\cdot q^{-1}, \; y\in \R^3, q\in S^3. \]
We then have  a homomorphism $\overline{Q}: S^3\to SO(3)$ which fits in an exact sequence: 
\[ 1\to \Z/2\Z=\langle -1\rangle\to S^3\stackrel{\overline{Q}}{\to} SO(3)\to 1.\]
We then have a commutative diagram: 
\[\begin{array}{ccc}
S^3 & \stackrel{\overline{Q}}{\longrightarrow}& SO(3)\\
\downarrow & & \downarrow \\
S^3\times S^3 & \stackrel{{Q}}{\longrightarrow} & SO(4)\\
\end{array}
\]
where the vertical maps are $\Delta:S^3\to S^3\times S^3$ and 
$J:SO(3)\to SO(4)$ given by $\Delta(x)=(x,x)$ and 
$J(A)=1\oplus A$. 

Therefore $\tilde G_2(\R^3)$ is identified with 
$\overline{Q}(S^3)/\overline{Q}(S^1)$ and 
$\tilde G_2(\R^4)$ with 
$Q(S^3\times S^3)/Q(S^1\times S^1)$. 
Observe that $\overline{Q}(S^3)/\overline{Q}(S^1)$ is 
homeomorphic to the quotient $S^3/S^1$, namely $S^2$ while 
$Q(S^3\times S^3)/Q(S^1\times S^1)$ is 
homeomorphic to the quotient $S^3\times S^3/S^1\times S^1$, namely $S^2\times S^2$. 
It follows that the map $\iota:S^2\to S^2\times S^2$ is covered by the 
diagonal $\Delta$ and hence it is identified with the diagonal inclusion. 

Once we have fixed the trivialization of $M^4$ and hence the 
identifications 
\[[N^3, \tilde G_2(\R^3)]= [N^3,S^2], \; 
[N^3, \tilde G_2(\R^4)]=[N^3,S^2]\times [N^3, S^2],\]
we derive that
\[\iota_{N^3}(f)=(f,f), \; f\in [N^3,S^2].\]

\begin{corollary}\label{shellindex2}
Let $X^4$ be a disk block whose generic fiber $F$ has empty boundary 
and $\xi(f)$ be the restriction of $\ker Df$ to $\partial X^4$. 
Let the characteristic class of the combing be 
$\mu_{N^3}(\xi)=\beta\in H^2(N^3)$. 
Then 
\[ (-\lambda(\xi),\rho(\xi))\in  \Z/d_{\beta}\Z\times  \Z/d_{\beta}\Z,\]
is an invariant of the shell equivalence class of $X^4$.  
\end{corollary}

An instructive  example is the case of a disk block $f:X^4\to D^2$ with trivial boundary monodromy, namely such that $\partial X^4=F\times S^1$. The 2-plane 
field $\xi(f)$ is the tangent plane at the $F$ factor and hence it is the pull-back 
of the tangent bundle $TF$ under the projection $F\times S^1\to F$. 
Let $[S^1]$ be the homology class of $S^1$ in $H_1(F\times S^1)$ and 
$PD$ denote the Poincar\'e dual. We derive that 
the Euler class of $\xi(f)$ is $\chi(f) PD[S^1]$. As $H^2(F\times S^1)$ is torsion free 
the class $\mu(\xi(f))=\beta$ is half the Euler class and so 
$d_{\beta}=|\chi(F)|$. Note that both $d_{\beta}$ and 
$(-\lambda(\xi(f)),\rho(\xi(f)))$ are independent of the ambient 
trivialization of $X^4$.

In particular, if the generic fiber $F$ of the disk block is a torus, 
then the invariants  $(-\lambda(\xi),\rho(\xi))\in  \Z\oplus  \Z$, are well-defined.

\begin{remark}
The previous results suggest a close relation between the index and the 
(refined) Gompf invariants of oriented 2-plane fields over $\partial X^4$ 
(see \cite{Gompf2,Kup}). 
\end{remark}


\subsection{Manifolds without singular fibrations}
\begin{proposition}\label{notfiber}
If $M^4$ is a closed orientable 4-manifold with $b_2(M)=0$ which admits a singular fibration over 
a closed oriented surface $N^2$, then $b_1(M)\leq 1$. In particular, $\varphi_s(M^4, N^2)=\infty$, whenever $b_1(M)\geq 2$ and $b_2(M)=0$. 
\end{proposition}
\begin{proof}[Proof of Proposition \ref{notfiber}]
If $H=0$ then $\sigma(M)=0$ and the Hirzebruch--Hopf theorem \ref{Hirzebruch-Hopf} above shows that the index 
of $\xi$ is $(e(M)/2, e(M)/2)$. From Rudolph's theorem \ref{Rudolph} we derive  
\[ \mu(f)=e(M)=2-2b_1(M).\]
Now the Milnor number of an isolated singularity is always non-negative and hence $b_1(M)\leq 1$, as claimed. 
\end{proof}

\begin{proposition}\label{notfibertorus}
If $M^4$ is a closed orientable 4-manifold with $b_2(M)=0$ and $b_1(M)\leq 1$ which admits a singular fibration over a closed oriented surface $N^2$, then either $N^2$ is a torus, or else 
the generic fiber $F$ is a torus. 
\end{proposition}
\begin{proof}
We can compute $e(M)$ by using the polyhedron $\Pi(f)$, as follows.
Let $D^2\subset N^2$ be a disk  containing all the critical values of $f$ and $X=f^{-1}(D^2)$. 
Then $M$ is the union of  $(N^2-{\rm int}(D^2))\times F$ and $X=f^{-1}(D^2)$ along $S^1\times F$. 
As $X$ retracts onto $\Pi(f)$, we derive that 
\[ e(X)=e(\Pi(f)).\]
As $\Pi(f)$ is obtained from $F$ by adding cones over subsurfaces $F_i$, 
\[ e(\Pi(f))=e(F)+\mu(f).\]
The additivity of the Euler characteristic for unions over $F\times S^1$ yields
\[ e(M)=e(N)e(F) +\mu(f).\]
We noted in the proof of proposition \ref{notfiber} that Theorem \ref{Hirzebruch-Hopf} along 
with  Theorem \ref{Rudolph} imply that 
\[ \mu(f)=e(M).\]
Therefore $e(F)e(N)=0$, as claimed. 
\end{proof}

\begin{proposition}
Let $M^4$ be a smooth closed oriented manifold of dimension $4$ with 
 $b_1(M)=1$ and $b_2(M)=0$
that admits some singular fibration $f:M^4\to S^2$.                                                                                          
Then $M^4$ is a  topological fibration by tori over $S^2$. 
\end{proposition}
\begin{proof}
We know that $\mu(f)=0$ by Proposition \ref{notfiber}. This implies that 
all local links are unknots and hence $f$ has no topological critical points.
Thus $M^4$ topologically fibers over $S^2$ and the fibers are tori. 
\end{proof}

An example as in the proposition above is $S^1\times S^3$ which fibers over $S^2$.

The following is a slight extension of \cite[Thm. 8.4.13]{GS}: 
\begin{proposition}
If $M^4$ is a smooth closed oriented 4-manifold with positive definite intersection pairing $S$ 
and $f:M^4\to N^2$ is a singular fibration over a closed oriented surface then:
\[ 1-b_1(M)+b_2(M) \geq \lambda(f), \, \rho(f)\geq 1-b_1(M).\]
\end{proposition}
\begin{proof}
This follows from Lemma \ref{classes} and the fact that $\Omega(M)\subseteq \sigma(M)+8\Z_+$, when 
$S$ is positive definite, since squares of characteristic elements are at least $\sigma(M)$. 
\end{proof}

\subsection{Singular fibrations with a single critical point revisited}
Let $X_{1,-1}$ be the preimage $f^{-1}(D^2)$ of  a disk containing the singular values, 
where $f$ denotes the Matsumoto achiral fibration 
by tori $f:S^4\to S^2$ from Section \ref{matsum}. Let $Y_n$ be the result of expanding the 
Milnor fiber of the pretzel link $(2,-2,2n)$, $n\in \Z$ to a torus. 
The boundary fibrations $f:\partial X_{1,-1}\to S^1$ and $g:\partial Y_n\to S^1$ are trivial. 
Out of the blocks $X_{1,-1}$ and $Y_n$ we can construct 4-manifolds by gluing to each of them a 
trivial fibration $S^1\times S^1\times D^2$ along their boundaries respecting the 
trivial fibrations. 
The result of gluing might depend on the element in  $\Z\oplus \Z$ corresponding to the homotopy class of the loop of gluing diffeomorphisms, as the monodromy map is not surjective as in the statement of  Proposition \ref{classificationg1}.  The associated singular fibrations will be called generalized Matsumoto fibrations. Note that all manifolds $M^4$ obtained  in this way are not necessarily homology spheres, as we can also obtain $S^2\times S^2$ from $X_{1,-1}$.

\begin{proposition}\label{hspheres}
Suppose that $M$ is a rational homology sphere of dimension $4$                                                                                       
that admits some singular fibration $f:M^4\to S^2$.                                                                                          
Then $M^4$ is homeomorphic to one of the manifolds constructed above from the 
blocks   $X_{1,-1}$ and $Y_n$ and $f$ is equivalent to the corresponding 
generalized Matsumoto singular fibrations by tori. 
\end{proposition}                                                                                                                            
\begin{proof}                                                                                                                                
We know that the generic fiber $F$ is a torus and $\mu(f)=2$.                                                                                
Since $\mu(f)_p$ is positive, $f$ has either two critical points $p_1,p_2$ with                                                                           
$\mu(f)_{p_i}=1$ or a single critical point with $\mu(f)_p=2$.                                                                               
                                                                                                                                             
Assume first that  $f$ has two critical points $p_i$ of                                                                          
$\mu(f)_{p_i}=1$. Then the local Milnor fiber $F_i$ of the singularity $p_i$ is                                                                   
an annulus. Therefore, the monodromy of the local Milnor fibration around $p_i$                                                              
has the form $T_{\gamma_i}^{k_i}$, where $\gamma_i$ is the core of the                                                                       
annulus $F_i$. Recall that $F_i$ are naturally embedded in $F$ and hence the simple                                                                   
curves $\gamma_i$ are drawn on $F$. Let $D^2\subset S^2$ be a disk containing both critical values $f(p_i)$ and $X=f^{-1}(D^2)$.  The monodromy of the restriction                                                                                                                    
$f|_{\partial X}:\partial X=f^{-1}(\partial D^2)\to \partial D^2$ should be trivial, as it bounds a trivial fibration over $S^2\setminus {\rm int}(D^2)$.                                                                                                                                
We then have the following relation in the mapping class group of the torus:                                                                 
\[ T_{\gamma_1}^{k_1} T_{\gamma_2}^{k_2}=1\in M(F)\]                                                                                         
                                                                                                                                             
We claim that $\gamma_1=\gamma_2$ and $k_1+k_2=0$. Indeed, according to \cite{Ishida}
given two essential simple closed curves $\gamma_i$ on the torus, the subgroup 
$\langle T_{\gamma_1}, T_{\gamma_2}\rangle$ generated 
by the two Dehn twists is:
\begin{enumerate}
\item  infinite cyclic when $\gamma_1$ and $\gamma_2$ are isotopic; 
\item free abelian freely generated by the two Dehn twists when $\gamma_1$ and $\gamma_2$ are disjoint up to isotopy;
\item a free group freely generated by the two Dehn twists when the intersection number of the 
isotopy classes of $\gamma_1$ and $\gamma_2$ is at least 2;
\item isomorphic to $SL(2,\Z)$ when the intersection number of the 
isotopy classes of $\gamma_1$ and $\gamma_2$ is 1. 
\end{enumerate}
In the case (4) the images of $T_{\gamma_1}$ and  $T_{\gamma_2}$ in $SL(2,\Z)$ act as opposite parabolics and so the image of $T_{\gamma_1}^{k_1} T_{\gamma_2}^{k_2}$ is not trivial unless 
$k_1=k_2=0$. Therefore only the case (1) can occur and this proves the claim.

Suppose now that $f$ has a single critical point with Milnor number $2$. Thus the local Milnor fiber is either a one-holed torus or a pair of pants.
 
It is well-known  that genus one fibered knots are either the trefoil (with either orientation) or the figure eight knot (which is amphicheiral), see \cite[Prop. 5.14]{BZ}.
There is a unique embedding of the one holed torus $F_1$ into the torus $F$. The monodromy of the open books associated with the 
trefoil knots and the figure eight knot can be found in \cite{BZ}. In both cases the corresponding matrices in $SL(2,\Z)$ are nontrivial and so 
the boundary fibration  $f|_{\partial X}:\partial X=f^{-1}(\partial D^2)\to \partial D^2$ is not trivial, contradicting our hypothesis. 
Thus this case cannot be realized. 

Further, we saw in Section \ref{examples} that the fiber of three component  fibered links of genus zero can be expanded to a torus if and only if  
these are the pretzel links $(2,-2,2n)$.                                                                                             
\end{proof}

Note that a homology sphere constructed out of $X_{1,-1}$ has an achiral Lefschetz fibration with two singular points of opposite signs and hence it must be homeomorphic to $S^4$.
However there are examples of rational homology spheres with nontrivial first homology 
which admit achiral Lefschetz fibrations.

\vspace{0.5cm}



\end{document}